\documentclass[12pt]{amsart}
\usepackage{amsmath}
\usepackage{amssymb}
\usepackage{amscd}
\usepackage{graphicx}
\newtheorem{theorem}{Theorem}[section]
\newtheorem{lemma}[theorem]{Lemma}
\newtheorem{proposition}[theorem]{Proposition}
\newtheorem{corollary}[theorem]{Corollary}

\theoremstyle{definition}
\newtheorem{definition}[theorem]{Definition}
\newtheorem{example}[theorem]{Example}
\numberwithin{equation}{section}

\newtheorem{remark}[theorem]{Remark}

\title{Mixed Frobenius Structure and local A-model}
\author{Yukiko Konishi}
\address{ 
Department of Mathematics, 
Kyoto University,
Kyoto 606-8502, Japan 
}
\email{konishi@math.kyoto-u.ac.jp}
\author{Satoshi Minabe}
\address{Department of Mathematics, 
Tokyo Denki University, 120-8551 Tokyo,  Japan}
\email{minabe@mail.dendai.ac.jp}
\begin{document}
\maketitle
\begin{abstract}
We define the notion of
mixed Frobenius structure which is a generalization
of the structure of a Frobenius manifold.  
We construct a mixed Frobenius structure on the cohomology of
weak Fano toric surfaces and that of  the three dimensional projective space 
using local Gromov--Witten invariants.
This is an analogue of the Frobenius manifold
associated to the quantum cohomology in the local
Calabi--Yau setting.
\end{abstract}
\section{Introduction}
The purpose of this paper is to introduce the notion of 
mixed Frobenius structure. It is a generalization of the 
structure of a Frobenius manifold, 
which plays an important role in the study of mirror 
symmetry. Our motivation for introducing it
comes from local mirror symmetry.
The main results, Theorems \ref{toric3} and \ref{prop-P3-2}, 
show that local Gromov--Witten invariants in 
the local A-model (i.e. the A-model side of local mirror symmetry)
give rise to mixed Frobenius structures.

In the rest of the introduction, we first explain the contents of the paper.
Then we explain the motivation from local mirror symmetry.
Throughout the paper, an algebra is an associative 
commutative algebra with unit over $\mathbb{C}$
of finite dimension. 
We denote by $\circ$ the multiplication of an algebra.
\subsection{A generalization of Frobenius algebra}
Recall that a Frobenius algebra is 
a pair of an algebra $A$ 
and a nondegenerate bilinear form $\langle~,~\rangle:
A\times A\to\mathbb{C}$
satisfying the compatibility condition
\begin{equation}\label{A1}
\langle x \circ y, z\rangle=\langle x,y\circ z \rangle\quad
(x,y,z\in A)~.
\end{equation}
For example, the even part of the 
cohomology ring of a compact oriented manifold and  
the intersection form is a Frobenius algebra (see, e.g. \cite{Kock}).

We generalize this notion as follows 
(\S \ref{seciton:Frobenius-filtration}). 
Let $A$ be an algebra. A Frobenius filtration on $A$ 
is a pair of an increasing sequence of ideals 
$\cdots \subset I_k\subset I_{k+1}\subset \cdots$ in $A$
and a set of  nondegenerate symmetric bilinear forms $(~,~)_k$ 
on graded quotients $I_{k}/I_{k-1}$ satisfying 
the condition similar to \eqref{A1}:
\begin{equation}\nonumber
( x,a\circ y)_k=( a\circ x,y)_k
 \qquad (x,y\in I_k/I_{k-1},~a\in A/I_{k-1})~.
\end{equation}
Given a   Frobenius algebra $A$ with a nilpotent element $n$,
there are constructions of the Frobenius filtration in
$A$ and the quotient algebra $A/(n)$ where $(n)$ is the ideal 
generated by $n$ (\S \ref{section:nilpotent-quotient-constructions}).
We call them the nilpotent construction and the quotient construction respectively.

To illustrate these constructions, we give two types of examples.
First examples are the cohomology rings 
of compact K\"ahler manifolds (\S \ref{example-cohomology}). 
Second examples are Chen--Ruan's cohomology rings 
of some non-compact orbifolds (\S \ref{example-orbifold}). 

\subsection{The mixed Frobenius structure}
Recall that
a Frobenius structure 
on a manifold  $M$ consists of 
a structure of the Frobenius algebra on the tangent bundle $TM$
and a vector field $E$ on $M$ called the Euler vector field
which satisfy certain compatibility conditions
(see Definition \ref{def:Frobenius-structure}).
It was defined by Dubrovin \cite{Dubrovin}
but before that
K. Saito found this structure in the singularity theory \cite{Saito} (see also \cite{Saito-Takahashi}).
An example of the Frobenius structure is the quantum cohomology ring 
(i.e. the cohomology with the quantum cup product)
of a compact symplectic manifold.   
For other examples, see e.g. \cite{Manin}
and references therein.

We generalize this notion and define the mixed Frobenius structure (\S \ref{section:MHS}).
A mixed Frobenius structure (MFS)
on a complex manifold $M$
consists of 
a structure of an algebra with a Frobenius filtration,
a torsion-free flat connection $\nabla$ on $TM$,
and an Euler vector field $E$
which satisfy compatibility conditions 
(see Definition \ref{def:MFS}).
We extend the nilpotent and the quotient constructions
to the MFS 
(\S \ref{section:nilpotent-quotient-construction2}).
If given a Frobenius manifold $M$ with a vector field $n$
which is nilpotent with respect to the multiplication
(and satisfies certain compatibility conditions with the metric and
the Euler vector field),
then we have a MFS
on $M$ and one on a certain submanifold 
(see Theorem \ref{thm:nilpotent2} and Corollary \ref{quotient-construction2}).

Finally, 
we apply the quotient construction to
the quantum cohomology ring of 
the projective compactification
$\mathbb{P}(K_S\oplus \mathcal{O}_S)$
where $S$ is a weak Fano toric surface
and $K_S$ is its  canonical bundle,
and obtain a MFS on $H^*(S, \mathbb{C})$
(\S \ref{example-toricsurface}).
We also apply the quotient construction to
$\mathbb{P}(K_{\mathbb{P}^3}\oplus \mathcal{O}_{\mathbb{P}^3})$
and obtain a MFS on  $H^*(\mathbb{P}^3, \mathbb{C})$
(\S \ref{example-localP3}).
%

In the appendix (\S \ref{section:deformed-connection}), 
we give a definition of the deformed connection 
and  show that it is flat. We also write down the 
deformed flat coordinates of the above MFS's. 

\subsection{Motivation from local mirror symmetry}
For  a Calabi--Yau threefold $X$,  
two Frobenius structures on $H^*(X, \mathbb{C})$ are known. 
The one is given by 
the quantum cohomology ring
and the intersection form (the A-model). 
The other is constructed by Barannikov--Kontsevith \cite{BK}
which is  closely related to the variation of Hodge structures 
on $H^3(X, \mathbb{C})$ (the B-model).
If Calabi--Yau threefolds $X$ and $Y$ are mirror partners,
it is conjectured that the former Frobenius structure on $H^*(X, \mathbb{C})$
is isomorphic to the latter Frobenius structure on $H^*(Y, \mathbb{C})$.
This conjecture was proved when $X$ is a complete intersection
in a projective space \cite{Barannikov}.

In \cite{CKYZ}, local mirror symmetry for 
weak Fano toric surfaces $S$
was derived from 
mirror symmetry of toric Calabi--Yau hypersurfaces
containing $S$ as a smooth divisor.
Therefore, looking at the A-model side, 
it is expected that 
$H^*(S, \mathbb{C})$ should inherit
a Frobenius structure from the quantum
cohomology of 
the corresponding Calabi--Yau threefold.

However, the above expectation turns out to be  too naive
because it seems that there is no natural way to obtain a nondegenerate
bilinear form on $H^*(S, \mathbb{C})$ from the intersection form on 
the Calabi--Yau threefold.
Therefore we have to abandon the nondegenerate pairing
and have to generalize the notion of Frobenius algebra. 
Hints on how come from looking at the B-model sides. 
In mirror symmetry, the B-model for Calabi--Yau manifolds 
is about the Hodge structure
whereas in local mirror symmetry it is about the mixed Hodge structure,
which has one more extra datum called the weight filtration.
This leads us to introduce the notion of Frobenius filtration.

We believe that the MFS on $H^*(S, \mathbb{C})$
constructed in 
\S \ref{example-toricsurface}
is what would be obtained from 
the Frobenius structure of the corresponding Calabi--Yau threefold
by the following two reasons.
The first reason is that 
in the case $S=\mathbb{P}^2$, if the 
quotient construction is applied to the 
quantum cohomology of the 
corresponding Calabi--Yau threefold $X$,
the resulting MFS is the same as this one.
The second reason is that the Frobenius filtration on $H^*(S, \mathbb{C})$
agrees with the weight filtration on the local B-model side
(see Remark \ref{rem:localB}).
It would be very interesting if we can construct 
a MFS on the local B-model side and see whether it is
isomorphic to the one constructed here.   
These are left as future problems.

\subsection*{Note added in 2018}
The original version of this paper was written in 2012.
Since then, there are some progresses on the study of mixed Frobenius structures.
In a sequel paper \cite{KonishiMinabe14}, the authors refined a definition of 
mixed Frobenius structures. The new definition requires a stronger condition
on the potentiality of the product. See \cite[Remark 4.7]{KonishiMinabe14}. 
It is shown in \cite{KonishiMinabe14} that the main results of this 
paper remain true under this new definition.
We also remark that mixed Frobenius structure on the local B-model side was
constructed by Shamoto \cite{Shamoto}. 

\subsection*{Acknowledgements}
The authors would like to thank Takuro Mochizuki for useful comments
on Kashiwara's filtration (see Remark \ref{rem:mochizuki}).
The first named author is supported in part by 
JSPS Grant-in-Aid for Young Scientists (No.\,22740013), 
Grant-in-Aid for Challenging Exploratory Research (26610008),
JSPS KAKENHI Kiban-A (16H02146) 
and JSPS KAKENHI Kiban-S (16H06337).
The second named author is supported in part by 
JSPS Grant-in-Aid for Young Scientists (No.\,24740028), 
Grant for Basic Science Research
Projects from The Sumitomo Foundation and JSPS KAKENHI Grand number JP17K05228. 
The authors thank the referee for useful comments. 
\section{Frobenius ideal and Frobenius filtration }
\label{seciton:Frobenius-filtration}
\begin{definition}
Let $A$ be an algebra.
A {\it Frobenius ideal} of $A$ is a pair $(I,~(~,~))$
of an ideal $I$ of $A$ and a nondegenerate symmetric bilinear form on $I$
which satisfies the condition
\begin{equation}\label{I1}
( x,a\circ y)=( a\circ x,y)
 \qquad (x,y\in I,~a\in A)~.
\end{equation}
\end{definition}

\begin{definition}
An increasing filtration of an algebra $A$ by ideals
\begin{equation}\nonumber
I_{\bullet}:  0\subset \cdots \subset~ I_k\subset~ I_{k+1}~\subset \cdots \subset A
\qquad (k\in \mathbb{Z})
\end{equation}
together with bilinear forms $(~,~)_k$ on $I_k/I_{k-1}$
is a {\it Frobenius filtration}
if $I_{\bullet}$ is exhaustive\footnote{
An increasing filtration $I_{\bullet}$ on 
a finite dimensional vector space
$A$ is exhaustive if
there exists $k, l$ such that 
$I_k=\{0\}$ and $I_l=A$.
} 
and 
$(I_k/I_{k-1},~(~,~)_k)$ is a Frobenius ideal of $A/I_{k-1}$.
\end{definition}



For the later purpose, we state the following
\begin{lemma}\label{lem:quotient}
If $(I_{\bullet},~(~,~)_{\bullet})$ is a Frobenius filtration
on an algebra $A$, then $(I_{\bullet}/I_k,~(~,~)_{\bullet})$
is a Frobenius filtration on the quotient algebra $A/I_k$
for any $k\in \mathbb{Z}$.
\end{lemma}

\section{Constructions of Frobenius filtration by a nilpotent element}
\label{section:nilpotent-quotient-constructions}
\subsection{Frobenius filtration defined by a nilpotent element}
\label{nilpotent-construction}
Let $(A,~~ \langle~~,~~\rangle)$ be a Frobenius algebra.
Assume that there exists a nilpotent element $n \in A$ of order $d$
(i.e. $n^d=0$, $n^{d-1}\neq 0$). 
For $0\leq k \leq d$, we set  $J_k:=\{ x \in A \mid x\circ n^k=0\}$. 
Then we have a filtration 
$$J_0=0\subset J_1 \subset \cdots \subset J_{d-1} \subset J_d=A$$
of ideals on $A$. Let $I=(n)$ be the ideal generated by $n$, and let
$I_k:=I+J_k$ ($0\leq k \leq d$). Then we have a filtration 
\begin{equation}\label{nilp1}
I_{-1}:=0\subset I_0=I \subset I_1 \subset \cdots \subset I_{d-1} \subset I_d=A
\end{equation}
of ideals in $A$.
Next
we define a pairing $(~~,~~)_k$ on each graded quotient $I_k/I_{k-1}$. 
Note that $I_k/I_{k-1}\cong J_k/ (J_{k-1}+I\cap J_k)$ for $k>0$.
\begin{definition}\label{nilp2}
(i) For $x, y \in I_0$, there are $\tilde{x}, \tilde{y}\in A$ such that 
$x=n\circ \tilde{x}$ and $y=n\circ \tilde{y}$. We define the pairing $(~~,~~)_0$ on $I_0$ by
$$(x, y)_0=\langle \tilde{x}, \tilde{y}\circ n \rangle~.$$
(ii) 
For $x, y\in I_k/I_{k-1}$ ($k>0$), we take representatives 
$\tilde{x},~\tilde{y} \in J_k$ and set 
$(x,y)_k:=\langle \tilde{x}, \tilde{y}\circ n^{k-1} \rangle$.
\end{definition}
It is easy to check that the pairing $(~~,~~)_k$ is well-defined.

\begin{lemma}
(i) The pairing $(~~,~~)_k$ on $I_k/I_{k-1}$ is symmetric and 
satisfies $$(a\circ x, y)_k=(x, a\circ y)_k$$
for any $x, y \in I_k/I_{k-1}$ and $a\in A/I_{k-1}$.\\
(ii) The pairing $(~~,~~)_k$ is nondegenerate.
\end{lemma}
\begin{proof}
The first statement follows from the Frobenius property \eqref{A1} 
of $\langle ~~, ~~\rangle$. 
For the second,  we first consider the case $k=0$. 
The pairing $(~~,~~)_0$ is nondegenerate since
\begin{eqnarray*}
(x,y)_0=0\quad \forall y\in I_0 \quad 
\Longleftrightarrow \langle \tilde{x}, n\circ \tilde{y}\rangle=0\quad \forall \tilde{y}\in A
&\Longleftrightarrow& x=n\circ \tilde{x}=0~. 
\end{eqnarray*}
Next we consider the case $k>0$.
Consider the pairing $\langle~~,~~\rangle_k$ on $A$ defined by
$\langle a,b\rangle_k:=\langle a,b\circ n^{k-1}\rangle$. 
Let $I^{\perp}:=\{a\in A \mid \langle a, b\rangle_k=0,~~\forall b\in I\}$.
Then it follows that $I^{\perp}=J_k$, since
\begin{eqnarray*}
\langle a, b\circ n^{k-1} \rangle=0, \quad \forall b \in I \quad
\Longleftrightarrow \quad
\langle a, \tilde{b}\circ n^{k} \rangle=0, \quad \forall \tilde{b} \in A \quad
\Longleftrightarrow \quad
a \in J_k~. 
\end{eqnarray*}  
The pairing  $\langle~~,~~\rangle_k$ is degenerate precisely along
$J_{k-1}$. Then the orthogonal of $I/(I\cap J_{k-1})$ in $A/J_{k-1}$ with respect to this
pairing is $J_k/J_{k-1}$. Therefore the orthogonal $K_k$ of $J_k$ in $A$ is 
$K_k=I+J_{k-1}$. It follows that $\langle~~,~~\rangle_k$ induces 
a nondegenerate pairing on $J_k/(K_k\cap J_k)=J_k/(J_{k-1}+I\cap J_k)$. It is clear that 
the induced pairing coincides with $(~~,~~)_k$ defined above. Thus 
the pairing  $(~~,~~)_k$ on $I_k/I_{k-1}$ is nondegenerate.
\end{proof}

To summarize, we have obtained the following 
\begin{proposition}\label{prop-nilp}
If $(A,\langle~,~\rangle)$ is a Frobenius algebra with a nilpotent element $n$,
$\left(I_{\bullet}\, ,~ (~~,~~)_{\bullet}\right)$ 
defined in \eqref{nilp1} and Definition \ref{nilp2}
is a Frobenius filtration on $A$.
\end{proposition}

\begin{remark}\label{rem:mochizuki}
If a linear endomorphism $n$ on a vector space $A$ is nilpotent
of order $d$,
then we have a filtration on $A$ called the monodromy weight filtration
\begin{equation}\label{MWF}
0\subset W_0\subset W_1\subset\cdots\subset W_{2d-1}\subset W_{2d-2}=A~
\end{equation}
uniquely determined by the conditions \cite[p.93]{Cattani}:
\begin{equation}\nonumber
\begin{split}
&n(W_l)~~\subset~~ W_{l-2}~,\\
&n^j~:~W_{d+j-1}/W_{d+j-2}~\stackrel{\sim}{\rightarrow}~
W_{d-j-1}/W_{d-j-2} \quad (0\leq j<d)~.
\end{split} 
\end{equation}
The filtration \eqref{nilp1} is related to the 
monodromy weight filtration by
\begin{equation}\nonumber
\begin{split}
I_0&=\mathrm{Im}\,n~, 
\\
I_k&
=\mathrm{Im}\,n+W_{k+d-2}
\quad (0<k\leq d)~.
\end{split}
\end{equation}
This filtration agrees with the filtration $(N_*A)_{\bullet}$ 
defined in \cite[\S 3.4]{Kashiwara86} for the filtration
$A_{\bullet}: 0=A_0\subset A_1=A$.
\end{remark}

\subsection{Frobenius filtration by quotient construction}
\label{quotient-construction}
As in \S \ref{nilpotent-construction},
let $(A,~\langle~,~\rangle)$ be a Frobenius algebra, $n\in A$
a nilpotent element of order $d$.
Let $J_k=\{x\in A\mid x\circ n^k=0\}$
and $I=(n)$.

By Lemma \ref{lem:quotient}, the Frobenius filtration 
$(I_{\bullet},~(~,~)_{\bullet})$
constructed in \S \ref{nilpotent-construction}
induces a  Frobenius filtration on the quotient algebra
$A'=A/I$. Explicitly, the induced sequence is 
\begin{equation}\label{quot-filtration}
I_{\bullet}':0~\subset~ I_1'\subset\cdots \subset~ I_d'~=~A'~
\qquad (I_k':=(I+J_k)/I)~.
\end{equation}
The induced  bilinear forms $( ~,~)'_k$ on $I_k'/I_{k-1}'$ 
are given by
\begin{equation}\label{quot-bilinear}
( a',b')_k'=( a,b )_k
\end{equation}
where $a,b\in J_k$ are those such that
the image under $J_k\hookrightarrow J_k+I\rightarrow I_k'$
are $a',b'\in I_k'$.

\begin{corollary}\label{prop-quotient}
If $(A,\langle~,~\rangle)$ is a Frobenius algebra with a nilpotent element
$n$,
$(I_{\bullet}'\,,~(~,~)_{\bullet}')$ defined in \eqref{quot-filtration},
\eqref{quot-bilinear}
is a Frobenius filtration on $A/I$.
\end{corollary}

\section{Examples I:  cohomology of compact K\"ahler manifolds}
\label{example-cohomology}
\subsection{Examples}
Let $X$ be a compact K\"ahler manifold of 
complex dimension $d-1$. Let $A=\oplus_{j=0}^{d-1}H^{2j}(X, \mathbb{C})$
be the even part of the cohomology of $X$. Then $A$ is a Frobenius algebra with 
respect to the cup product $\cup$ and the usual intersection pairing 
$$
\langle x , y \rangle =
\int_{{X}}x \cup y
\qquad (x, y \in H^*({X}, \mathbb{C})).
$$
Let $n\in H^2(X, \mathbb{C})$ be a class such that 
either $n$ or $-n$ is a K\"ahler class.  
Then $n$ is nilpotent of order $d$. 
By the construction in \S \ref{nilpotent-construction},
we get a Frobenius filtration
$$0\subset I_0 \subset I_1 \subset \cdots \subset I_{d-1} \subset I_d=A~.$$

\begin{example}\label{example-Pn}
Let $X=\mathbb{P}^{d-1}$, $L=\mathcal{O}_{\mathbb{P}^{d-1}}(m)$
($m\in \mathbb{Z}$, $m\neq 0$) and $n=c_1(L)$. 
Then $A:=H^*(X,\mathbb{C})\cong \mathbb{C}[n]/(n^d)$
and the intersection form is given by 
$\langle n^i,n^j\rangle=m^{i+j}\delta_{i+j,d-1}$.
The Frobenius filtration on $A$
defined by $n$ is given as follows. The filtration is 
$$I_0=\bigoplus_{i=1}^{d-1} \mathbb{C}n^i,\qquad I_k=I_0 \quad (1\leq k \leq d-1),\qquad I_d=I_0\oplus \mathbb{C}1~.$$
If we set $e_i:=n^i/m^i$ ($1\leq i \leq d-1$) and $e_{d}:=1$, 
the pairings are
$$(e_i, e_j)_0=\frac{1}{m}\delta_{i+j, d}~,\qquad (e_d, e_d)_d=m^{d-1}~.$$ 
\end{example}

\begin{example}\label{example-surface}
Let $(X, L)$ be a polarized algebraic surface and $n=c_1(L)$. 
Then by the Lefschetz decomposition, we have
$$H^2(X, \mathbb{C})=H^2_{\rm prim}(X, \mathbb{C}) \oplus \mathbb{C}n~,$$
where $H^2_{\rm prim}(X, \mathbb{C})$ is the kernel of 
$n\, \cup : H^2(X,\mathbb{C}) \to H^4(X, \mathbb{C})$. 
The Frobenius filtration on $A=H^{\rm even} (X, \mathbb{C})$ 
defined by $n$ is given as follows. The filtration is:
\begin{eqnarray*}
I_0&=&\mathbb{C}n\oplus  \mathbb{C}n^2~,
\\
I_1&=&I_0\oplus H^2_{\rm prim}(X, \mathbb{C})~,
\\
I_2&=&I_1~,
\\
I_3&=&I_1\oplus \mathbb{C}1~.
\end{eqnarray*}
The pairings are:
$$(n, n^2)_0=K,\qquad (n,n)_0=(n^2,n^2)_0=0~,$$
$$(x ,y)_1=\int_Xx \cup y~,$$
$$(1,1)_3=K~,$$
where $K=\int_X n^2$.
\end{example}
\subsection{Remarks on a mixed Hodge structure for Example \ref{example-surface}}
\label{sec:MHS}
In this subsection, we explain that the filtration 
$I_{\bullet}$ in Example \ref{example-surface} 
can be regarded as a weight filtration of an $\mathbb{R}$-mixed Hodge structure.
See e.g. \cite[\S 3.1]{PetersSteenbrink} for a definition of the $\mathbb{R}$-mixed Hodge structure.
Although the result of this subsection is not used in the rest of the paper, 
it is relevant to local mirror symmetry. See Remark \ref{rem:localB}.

Let $c\in \sqrt{-1}\mathbb{R}$ be a nonzero purely imaginary number.
Let $X$ and $n$ be as in Example \ref{example-surface}.
We set
\begin{equation}\nonumber
\begin{split}
A_{\mathbb{R}}&=\{x\cup e^{ cn}\mid x\in H^{even}(X,\mathbb{R}) \}
\\
&=\mathbb{R}\Big(1+cn+\frac{c^2}{2}n^2\Big)\oplus
\mathbb{R}(n+cn^2)\oplus H_{\mathrm{prim}}^2(X,\mathbb{R})
\oplus \mathbb{R}n^2~.
\end{split}
\end{equation}
We consider the increasing filtration
$\mathcal{W}_{\bullet}$ on $A_{\mathbb{R}}$ given by
\begin{equation}\nonumber
0\subset 
\mathcal{W}_1=\mathbb{R}(n+cn^2)\oplus \mathbb{R}n^2\subset
\mathcal{W}_2=\mathcal{W}_1\oplus H^2_{\mathrm{prim}}(X,\mathbb{R})
=\mathcal{W}_3
\subset
\mathcal{W}_4=A_{\mathbb{R}}~.
\end{equation}
Notice that $\mathcal{W}_{k}\otimes \mathbb{C}=I_{k-1}$ holds.
We also consider the decreasing filtration $\mathcal{F}^{\bullet}$
on $A_{\mathbb{R}}\otimes \mathbb{C}=A$ given by
the degree of the cohomology:
\begin{equation}\nonumber
0\subset \mathcal{F}^{2}=H^0(X,\mathbb{C})\subset
\mathcal{F}^1=H^0(X,\mathbb{C})\oplus H^2(X,\mathbb{C})
\subset
\mathcal{F}^0=A~.
\end{equation}
Then it is not difficult to check the following
\begin{proposition}
(i) $(\mathcal{W}_{\bullet},\mathcal{F}^{\bullet})$
is an $\mathbb{R}$-mixed Hodge structure on $A_{\mathbb{R}}$.
\\
(ii) If we set
$$A^{p,q}=\mathcal{F}^p({\rm Gr}_k^{\mathcal{W}}A)
\cap \overline{\mathcal{F}^q({\rm Gr}_k^{\mathcal{W}}A)} \quad (p+q=k),$$
then the Hodge decomposition $A\cong \oplus_{1\leq k \leq 4} \oplus_{p+q=k} A^{p,q}$ 
is given by the following table.
\begin{equation}\nonumber
\begin{array}{|r|r|c|c|}\hline
 A^{p,k-p}    &p=2&1 &0\\\hline
k=1& &\mathbb{C}n &\mathbb{C}(n^2+\frac{1}{2c}n)\\
2     & 0&H^2_{\rm prim}&0\\
3     &0&0&0\\
4     & \mathbb{C}1&0&0\\\hline 
\end{array}~~~~
\end{equation}
\end{proposition}

Next we explain that the pairings $(~~,~~)_{\bullet}$ induce 
a polarization on the above $\mathbb{R}$-mixed Hodge structure.
We assume the following conditions:
\begin{equation}\label{eq:conditions}
\begin{cases}
h^{2,0}(X)=h^{0,2}(X)=0, \\
b_1(X) ~~\mathrm{is~~even}, \\
\int_X n^2=K>0.
\end{cases}
\end{equation}
Then the pairing $\langle~~,~~\rangle$ is negative definite on $H^2_{\rm prim}(X, \mathbb{R})$.
See e.g. \cite[\S IV, Corollary 2.15]{BHPV}. 

For simplicity, we set $c=\sqrt{-1}$.
Let us define $(-1)^k$-symmetric bilinear form
$Q_k: {\rm Gr}_k^{\mathcal{W}}A_{\mathbb{R}} \times {\rm Gr}_k^{\mathcal{W}}A_{\mathbb{R}} \to \mathbb{R}$ by
$$
Q_k(x,y)=\frac{(\sqrt{-1})^k}{2}\left\{ 
(Cx\, ,\, y)_{k-1}+(-1)^k(x\, ,\, Cy)_{k-1}
\right\},
$$
where $C|_{A^{p,q}}=(\sqrt{-1})^{p-q}$ is the Weil operator.
Explicitly, $Q_k$ is given as follows:
$$
Q_1(n^2\, ,\, n+\sqrt{-1}n^2)=K,\qquad Q_1(n+\sqrt{-1}n^2\,, \, n+\sqrt{-1}n^2)=Q_1(n^2\, ,\, n^2)=0,
$$
$$
Q_2(x,y)=-\int_X  x\cup y~,
$$
$$
Q_4(1,1)=K~.
$$
Then the Hermitian forms $H_k(x,y):=Q_k(Cx,\overline{y})$ are positive definite, and we have
the following
\begin{proposition}
Under the conditions \eqref{eq:conditions}, 
$(\mathcal{W}_{\bullet}, \mathcal{F}^{\bullet}, Q_{\bullet})$ is a graded polarized 
$\mathbb{R}$-mixed Hodge structure on $A_{\mathbb{R}}$.
\end{proposition}

\section{Examples II:  cohomology of some non-compact orbifolds}
\label{example-orbifold}
Let $\mathcal{X}$ be an algebraic orbifold.
We denote by  $H^*_{\rm orb}(\mathcal{X}, \mathbb{C})$ 
the orbifold cohomology ring of $\mathcal{X}$ introduced by Chen--Ruan \cite{CR}.
If $\mathcal{X}$ is compact, $H^*_{\rm orb}(\mathcal{X}, \mathbb{C})$
has the Poincar\'e duality pairing $\langle~,~\rangle$ 
given by  
$$
\langle x , y \rangle =
\int_{\mathcal{X}}x \cup_{\rm orb} y
\qquad (x, y \in H^*_{\rm orb}(\mathcal{X}, \mathbb{C})),
$$
where $\cup_{\rm orb}$ denotes Chen--Ruan's product.
It makes  $H^*_{\rm orb}(\mathcal{X}, \mathbb{C})$ into a Frobenius algebra.
See \cite[Theorem 4.6.6]{CR}.
The following examples can be computed by the results of \cite{BCS, Jiang, Mann}. 

\begin{example}
Let $\mathcal{X}=[\mathbb{C}^d/\mathbb{Z}_d]$ be the quotient  orbifold 
of type $\frac{1}{d}(1,\ldots, 1)$ (in the notation of \cite{Reid})
and $\overline{\mathcal{X}}=\mathbb{P} (1,\ldots, 1, d)$
be the $d$-dimensional weighted projective space of weight $(1,\ldots, 1, d)$.
The latter is a compactification of the former.  The divisor 
$D=\overline{\mathcal{X}}\setminus \mathcal{X}$ is 
Poincar\'e dual to $dH \in H^2_{\rm orb} (\overline{\mathcal{X}}, \mathbb{C})$,
where $H:=c_1(\mathcal{O}_{\mathbb{P} (1,\ldots, 1, d)}(1))$.
If we denote by $E \in H^2(\overline{\mathcal{X}}, \mathbb{C})$
the image of the unit class on the twisted sector $\mathbb{P}(d)$ of age $1$ 
of the inertia orbifold of $\overline{\mathcal{X}}$,   
then $H^*_{\rm orb}(\overline{\mathcal{X}}, \mathbb{C})$
is isomorphic to $\mathbb{C}[H, E] / (H^d-E^d, HE)$
as an algebra. The pairings on $H^*_{\rm orb}(\overline{\mathcal{X}}, \mathbb{C})$ are
determined by 
\begin{equation}\nonumber
\begin{split}
\langle H^i , H^j\rangle 
&= \langle E^i , E^j \rangle
=\frac{1}{d}\delta_{i+j,d}\quad(0\leq i,j\leq d)~,\\
\langle H^i,E^j\rangle&=0\quad (1\leq i,j\leq d)~.
\end{split}\end{equation}
(i) First, we apply the nilpotent construction 
to $H^*_{\rm orb}(\overline{\mathcal{X}}, \mathbb{C})$
with the nilpotent element $n:=dH$ of order $d+1$. Then
we obtain the following Frobenius filtration on $H^*_{\rm orb}(\overline{\mathcal{X}}, \mathbb{C})$.
The filtration  $I_{\bullet}$ of ideals is
$$I_0=(H) \subset I_1=\cdots =I_d=(H ,E) \subset  I_{d+1}=H^*_{\rm orb}(\overline{\mathcal{X}}, \mathbb{C})~.$$
The pairings $(~~,~~)_{\bullet}$ on graded quotients are
$$
(H^i,H^j)_0=\frac{1}{d^2}\delta_{i+j,d+1}, \quad
(E^i, E^j)_1=\frac{1}{d}\delta_{i+j,d}, \quad
(1,1)_{d+1}=d^{d-1}.
$$
(ii) Next, we take the quotient of $H^*_{\rm orb}(\overline{\mathcal{X}}, \mathbb{C})$
by the ideal $I_0=(n)$. The quotient algebra is isomorphic to 
$H^*_{\rm orb}({\mathcal{X}}, \mathbb{C})\cong \mathbb{C}[E]/(E^d)$.
The induced Frobenius filtration $\left(I_{\bullet}', (~~,~~)'_{\bullet} \right)$  on the quotient is 
identical to the one on 
$H^*(\mathbb{P}^{d-1})$ given in Example \ref{example-Pn}
with $m=d$.
\end{example}

\begin{example}
Let $\mathcal{X}=[\mathbb{C}^3/\mathbb{Z}_4]$ be the quotient  orbifold 
of type $\frac{1}{4}(1,1, 2)$ 
and $\overline{\mathcal{X}}=\mathbb{P} (1, 1, 2, 4)$.
The latter is a compactification of the former.  The divisor 
$D=\overline{\mathcal{X}}\setminus \mathcal{X}$ is 
Poincar\'e dual to $4H \in H^2_{\rm orb} (\overline{\mathcal{X}}, \mathbb{C})$,
where $H:=c_1(\mathcal{O}_{\mathbb{P} (1,1, 2, 4)}(1))$.
If we denote by $E_1$ (resp. $E_2$) $\in H^2(\overline{\mathcal{X}}, \mathbb{C})$
the image of the unit class on the twisted sector $\mathbb{P}(4)$
(resp. $\mathbb{P}(2,4)$)  of age $1$ of the inertia orbifold of $\overline{\mathcal{X}}$, 
then $H^*_{\rm orb}(\overline{\mathcal{X}}, \mathbb{C})$
is isomorphic to $\mathbb{C}[H, E_1, E_2] / (H^2-E_2^2, HE_1, 2HE_2=E_1^2)$
as an algebra. The pairings on $H^*_{\rm orb}(\overline{\mathcal{X}}, \mathbb{C})$
are determined by 
$$
\langle H^i , H^j\rangle =\frac{1}{8}\delta_{i+j,3}
$$
together with  the relations in the algebra and the Frobenius property \eqref{A1}.\\
(i) Applying the nilpotent construction 
to $H^*_{\rm orb}(\overline{\mathcal{X}}, \mathbb{C})$
with the nilpotent element $n:=4H$ of order $4$, we obtain the 
following Frobenius filtration on $H^*_{\rm orb}(\overline{\mathcal{X}}, \mathbb{C})$.
The filtration  $I_{\bullet}$ of ideals is
$$I_0=(H) \subset I_1=(H ,E_1) \subset  
I_2=I_3=(H, E_1, E_2)\subset
I_4=H^*_{\rm orb}(\overline{\mathcal{X}}, \mathbb{C})~.$$
The pairings $(~~,~~)_{\bullet}$ on graded quotients are
$$
(H^i,H^j)_0=\frac{1}{32}\delta_{i+j,4}, \quad
(E_1, E_1E_2)_1=\frac{1}{4}, \quad
(E_2, E_2)_2=\frac{1}{2}, \quad
(1,1)_{4}=8.
$$
(ii) Taking the quotient of $H^*_{\rm orb}(\overline{\mathcal{X}}, \mathbb{C})$
by the ideal $I_0=(n)$, we obtain the following Frobenius filtration on  
$H^*_{\rm orb}({\mathcal{X}}, \mathbb{C})\cong \mathbb{C}[E_1, E_2]/(E_1^2, E_2^2)$:
$$
I_1'=\mathbb{C}E_1\oplus \mathbb{C}E_1E_2 \subset
I_2'=I_3'=I_1'\oplus \mathbb{C}E_2\subset
I_4'=H^*_{\rm orb}({\mathcal{X}}, \mathbb{C}).
$$
\end{example}

\section{The mixed Frobenius structure}
\label{section:MHS}
In this section, a manifold is a complex manifold.
The holomorphic tangent bundle (resp. the holomorphic cotangent bundle)
of a manifold $M$ is denoted by $TM$ (resp. $T^*M$).
All vector bundles on $M$ are assumed to be holomorphic.
The space of  holomorphic sections on an open set $U\subset M$ 
of a vector bundle $\mathbb{E}\to M$ is 
denoted by $\Gamma(U,\mathbb{E})$. 
We mean by $\Gamma(\mathbb{E})$ $\Gamma(U,\mathbb{E})$ for some open subset 
$U\subset M$. 
The dual vector bundle of $\mathbb{E}$ is denoted by $\mathbb{E}^{\vee}$.

\subsection{Preliminary}
We say that a subbundle $I\subset TM$ is {\it $E$-closed}
for a given global vector field $E$ on $M$
if 
\begin{equation}
[E,x]~\in~\Gamma(I)~\quad (x~\in~\Gamma(I))~. 
\end{equation}
Notice that when this holds, the Lie bracket
$[E,\ast]$ induces a derivation on 
the quotient bundle $TM/I$.

We say that a subbundle $I\subset TM$ is {\it $\nabla$-closed}
for an affine connection $\nabla$ if
\begin{equation}
\nabla_z \,x~\in~\Gamma(I)~\quad 
( x~\in \Gamma(I)~,
~~z~\in~\Gamma(TM))~.
\end{equation}
Notice that when this holds, $\nabla$ induces a connection on $TM/I$.

\subsection{The Mixed Frobenius Structure}
Recall the definition of the Frobenius structure \cite{Dubrovin}.
 
\begin{definition}\label{def:Frobenius-structure}
A Frobenius structure of  charge $D\in \mathbb{C}$ on a manifold $M$
consists of a structure of the Frobenius algebra 
$(A_t,\langle~,~\rangle_t)$
on each tangent space $T_t M$ ($t\in M$) depending complex analytically on $t$,
and a globally defined vector field $E$ on $M$ called the Euler vector field
satisfying the following conditions.
\begin{itemize}
\item The Levi--Civita connection $\nabla$ of the metric $\langle~,~\rangle$ is flat
and the unit vector field $e$ is $\nabla$-flat (i.e. $\nabla e=0$).
\item
The bundle homomorphism $c: TM^{\otimes 3}\to \mathcal{O}_M$
defined by $c(x,y,z)=\langle x,y\circ z\rangle$ satisfies
\begin{equation}\label{c1}
(\nabla_w c) (x,y,z)=(\nabla_z c)(x,y,w)~,
\end{equation}
where $\nabla$ is the induced connection on $T^*M^{\otimes 3}$.\footnote{
This condition together with the Frobenius property \eqref{A1}
implies that $\nabla c: TM^{\otimes 4} \to \mathcal{O}_M$ is symmetric.}
\item The Euler vector field $E$ satisfies $\nabla\nabla E=0$ 
(here the leftmost $\nabla$  is the induced connection on $\mathrm{End}\,(TM)$)
and 
\begin{eqnarray}
\label{E1}
&[E,x\circ y]-[E,x]\circ y-x\circ[E,y]=x\circ y \quad 
(x,y\in \Gamma(TM))~,
\\
\label{E3}
&E\langle x,y\rangle-
\langle\, [E,x],y\rangle-
\langle x,[E,y]\,\rangle
 =(2-D)\langle x,y\rangle~\quad
(x,y\in \Gamma(TM))~.
\end{eqnarray}
\end{itemize}
\end{definition}

Now we generalize the Frobenius structure
to incorporate the Frobenius filtration.

\begin{definition}\label{def:MFS}
A mixed Frobenius structure $(\nabla,~E,~\circ, ~I_{\bullet},~(~,~)_{\bullet})$
on a manifold $M$ of  reference charge $D\in \mathbb{C}$
consists of 
\begin{itemize}
\item a  torsion free, flat connection $\nabla$ on the tangent bundle $TM$,
\item a global vector field $E$ satisfying $\nabla\nabla E=0$ 
called the Euler vector field,
\item
a fiberwise multiplication on every tangent space $T_t M$
depending complex analytically 
on $t$ such that  the unit vector field $e$ is $\nabla$-flat,
\item 
a Frobenius filtration 
$(I_{\bullet,t}~,~~ (~,~)_{\bullet,t})$
on every tangent space $T_t M$ depending complex analytically on $t$
such that $I_k$ are $\nabla$-closed and
$E$-closed subbundles of $TM$.
\end{itemize}
These must
satisfy the following 
compatibility  conditions.
Let $\pi_k:TM\to TM/I_{k-1}$ be the quotient map and let
$\circ_k$ be the induced multiplication on the quotient bundle $TM/I_{k-1}$.
Let $[E,\ast]_k$ and $\nabla^{(k)}$ 
be the derivation and the connection on $TM/I_{k-1}$ induced from $[E,\ast]$
and $\nabla$.
\begin{itemize}
\item The connection $\nabla$ and the bilinear forms $(~,~)_k$ must be 
compatible in the sense that
\begin{equation}\label{eta-const}
z( x,y)_k=(\nabla^{(k)}_z x,y)_k+
 ( x,\nabla^{(k)}_z y )_k~\quad
(x,y\,\in\,\Gamma(I_k/I_{k-1}),~
z\,\in\, \Gamma(TM))~.
\end{equation}
\item 
 The vector bundle homomorphism $c_k: I_k/I_{k-1} \otimes I_k/I_{k-1}\otimes
 TM \rightarrow \mathcal{O}_M$
defined by
\begin{equation}
c_k(x,y,z):=( x, ~\pi_k(z)\circ_k y )_k~,
\end{equation}
must satisfy
\begin{equation}\label{c-symmetry}
(\nabla^{(k)}_w c_k)(x,y,z)=(\nabla^{(k)}_z c_k)(x,y,w)~\quad
(x,y\in \Gamma(I_k/I_{k-1}),
z,w\in \Gamma(TM))~.
\end{equation}
Here $\nabla^{(k)}$ is the induced connection on 
$(I_k/I_{k-1})^{\vee}\otimes (I_k/I_{k-1})^{\vee}\otimes T^*M$.
\item The Euler vector field $E$ must satisfy
\begin{eqnarray}
\label{E-multiplication}
&[E, x\circ_k \pi_k(z)]_k-x\circ_k\pi_k([E,z])
 -[E,x]_k\circ_k\pi_k(z)=x\circ_k\pi_k(z)~,
\\
\label{E-metric}
&E (x,y)_k-( [E,x]_{k},y)_k
-( x,[E,y]_{k})_k=(2-D+k)( x,y)_k~
\\
&\nonumber\qquad
(x,y\,\in\,\Gamma(I_k/I_{k-1}),~
z\,\in\, \Gamma(TM))~.
\end{eqnarray}
\end{itemize}
\end{definition}

\subsection{Flat coordinates}
We write the conditions for the mixed Frobenius structure
in a local coordinate expression.

Let $m_k$ ($k\in\mathbb{Z}$) be the rank of $I_k/I_{k-1}$.
The flatness, the
torsion-free condition for $\nabla$
and the $\nabla$-closedness of $I_k$'s imply that
there exists on $M$ a system of local coordinates 
$t^{ka}$ ($k\in \mathbb{Z}$, $1\leq a\leq m_k$)
which satisfies the following two conditions:
$$
\nabla \frac{\partial}{\partial t^{ka}}=0 \qquad 
(k\in \mathbb{Z},~1\leq a\leq m_k)~,
$$
\begin{equation}\label{flat-frame}
\left \{
\frac{\partial}{\partial t^{la}} ~{\Big |}~ l\leq k, 1\leq a\leq m_l
\right\}\quad
\text{is a local frame of $I_k$.}
\end{equation}
Now assume that  one such system of flat coordinates is fixed.
Then we can naturally regard $
\{\frac{\partial}{\partial t^{ka}} \}_{1\leq a\leq m_k}
$ as 
a local frame of 
$I_k/I_{k-1}$.
We sometimes use the shorthand notation  $\partial_{ka}=\frac{\partial}{\partial t^{ka}}$.

Let $\eta^{(k)}$ be the matrix representation of $(~,~)_k$: 
\begin{equation}\nonumber
\eta_{ka,kb}=\Big( \frac{\partial}{\partial t^{ka}}\, ,\,
 \frac{\partial}{\partial t^{kb}}\Big)_k
\qquad (1\leq a,b\leq m_k)~,\quad
\eta^{(k)}:=(\eta_{ka,kb})~.
\end{equation}
Let $C_{{ka},{lb}}^{jc}$ be the structure constant of the multiplication:
\begin{equation}\nonumber
\frac{\partial}{\partial t^{ka}}\circ
\frac{\partial}{\partial t^{lb}}=
\sum_{j\in \mathbb{Z},\,1\leq c\leq m_j} 
C_{{ka},{lb}}^{jc}\,\frac{\partial}{\partial t^{jc}}~.
\end{equation}
Let us write the Euler vector field $E$ as
\begin{equation}\label{defEka}
E=\sum_{k\in\mathbb{Z},\,1\leq a\leq m_k}
E^{ka}\, \frac{\partial}{\partial t^{ka}}~.
\end{equation}
Then the conditions are summarized as follows.
(We omit the associativity and the commutativity conditions.)
\begin{itemize}
\item 
$\eta^{(k)}$ is a symmetric invertible matrix 
since $(~,~)_k$ is symmetric
and nondegenerate. It is a constant matrix by (\ref{eta-const}).
\item 
The condition that $I_k$ is an ideal is  equivalent to
$C_{ka,lb}^{jc}=0$ if $k<j$  or $l<j$.
The compatibility of $(~,~)_k$ and the multiplication (see \eqref{I1})
is equivalent
to 
\begin{equation}
\sum_{1\leq d\leq m_k}C_{lc,kb}^{kd}\eta_{kd,ka}
=\sum_{1\leq d\leq m_k}C_{lc,ka}^{kd}\eta_{kd,kb}~.
\end{equation}
\item The condition \eqref{c-symmetry} (with the nondegeneracy of $\eta^{(k)}$)
is equivalent to
\begin{equation}\label{c-symmetry2}
\begin{split}
&\frac{\partial}{\partial {t^{jd}}}C_{ka,lc}^{kb}=
\frac{\partial}{\partial {t^{lc}}}C_{ka,jd}^{kb}
\qquad(l,j\geq k)~,
\\
&\frac{\partial}{\partial {t^{jd}}}C_{ka,lc}^{kb}=0
\qquad(j<k)~.
\end{split}
\end{equation}
\item
$\nabla\nabla E=0$ and  $E$-closedness of $I_k$'s imply
\begin{equation}\label{E-linear}
\frac{\partial^2}{\partial t^{jc} \partial t^{lb}}\,
E^{ka}=0\quad ({}^{\forall} l,j)~,\quad
\frac{\partial}{\partial t^{lb}}\, E^{ka}=0 \quad (l<k)~.
\end{equation}
\item
Eq.\eqref{E-multiplication} together with \eqref{c-symmetry2}, \eqref{E-linear}
is equivalent to
\begin{equation}\label{E-multiplication2}
\sum_{\begin{subarray}{c}j\geq k,\\1\leq d\leq m_j\end{subarray}}
\partial_{la}(E^{jd}C_{jd,kb}^{kc})
-\sum_{1\leq d\leq m_k}C_{la,kb}^{kd}(\partial_{kd}E^{kc})
+\sum_{1\leq d\leq m_k}C_{la,kd}^{kc}(\partial_{kb}E^{kd})
=C_{la,kb}^{kc}~.
\end{equation}
\item
Eq.\eqref{E-metric} is equivalent to
\begin{equation}\label{E-metric2}
\sum_{1\leq c\leq m_k} \big( \eta_{kb,kc}
(\partial_{ka}E^{kc})+\eta_{ka,kc}(\partial_{kb}E^{kc})\big)
=(2-D+k) \, \eta_{ka,kb}~.
\end{equation}
\end{itemize}

\subsection{Mixed Frobenius structure on a transversal slice.}
\label{sec:transversal}
Let $(\nabla,~E,~\circ, ~I_{\bullet},~(~,~)_{\bullet})$ be a
mixed Frobenius structure of reference charge $D$ on $M$.
Let $m_k$ be the rank of $I_k/I_{k-1}$.
Fix a system of flat coordinates $t^{ka}$ 
($k\in \mathbb{Z}$, $1\leq a\leq m_k$) on $M$
(see \eqref{flat-frame}).
Since each subbundle $I_k\subset TM$ is involutive,
it defines a foliation on $M$. Let $M_k$ be a leaf, i.e.
in a neighborhood $U\subset M$ 
where the coordinates $t^{ka}$ are well-defined,
\begin{equation}\label{eq:leaf}
M_k\cap U=\{t^{la}=\text{constant}\mid l>k\}~.
\end{equation}
Let $$
M^{(k+1)}=\{t^{la}=\text{constant}~\mid~ l\leq k,~1\leq a\leq m_l \} \subset 
U~.
$$
This  is transversal to the leaf $M_k$.
Then, locally on $U$,
we have
the direct sum decomposition $TM=TM^{(k+1)}\oplus I_k$
and 
obtain
the isomorphism $TM^{(k+1)}\cong TM/I_k$.
Let $E^{(k+1)}$ be the vector field on $M^{(k+1)}$ induced from 
the Euler vector field $E$: if we use  the flat coordinate expression 
in \eqref{defEka},
$$
E^{(k+1)}=\sum_{\begin{subarray}{c}l\geq k+1,\\
1\leq a\leq m_l\end{subarray}}
E^{la}\, \frac{\partial}{\partial t^{la}}~
\quad \Big(\text{just drop the terms $\frac{\partial}{\partial t^{ma}}$ for $m\leq k$ in $E$} \Big).
$$
This is a well-defined vector field on $M^{(k+1)}$ since 
$E^{la}$ ($l\geq k+1$) is independent of $t^{mb}$ ($m\leq k$), see \eqref{E-linear}.

\begin{lemma}\label{lem:quotient-construction2}
$(\nabla^{(k+1)},~E^{(k+1)},~\circ_{k+1}, ~I_{\bullet}/I_k,~(~,~)_{\bullet})$ is a
mixed Frobenius structure of reference charge $D$ on $M^{(k+1)}$.
\end{lemma}

\section{Construction by a nilpotent vector field}
\label{section:nilpotent-quotient-construction2}
\subsection{Construction by a nilpotent vector field}
\label{ffm}

Let $(\circ,~\langle~,~\rangle,~E)$ be a Frobenius structure on a
manifold $M$ of  charge $D$. Let $\nabla$ be the Levi--Civita connection of 
the metric $\langle ~,~\rangle$.
Assume that there exists
a global vector field $n$  satisfying the following conditions:
\begin{eqnarray}
\nonumber
&n^d=\underbrace{n\circ n\circ \cdots \circ n}_{\text{$d$ times}}= 0~,\quad n^{d-1}\neq 0~,\\
&\label{condition-E-n}
[E,n]=0~,
\\
&\label{condition-nabla-n}
\nabla(n\circ x)=n\circ \nabla x~\quad(x\in \Gamma(TM))~.
\end{eqnarray}

The condition \eqref{condition-nabla-n} implies that 
the map $n\, \circ :TM\to TM$ is a flat bundle homomorphism.
So the ranks of the kernel and the image of 
$n\, \circ$ are constant.
As in \eqref{nilp1},
we define
\begin{equation}\nonumber
I=\mathrm{Im}\,(n\, \circ)~, \quad
J_k=\mathrm{Ker}\,(n^k\, \circ) \quad (1\leq k\leq d)~.
\end{equation}
Let $I_0=I$, $I_k=I+J_k$.

\begin{lemma} $I,J_k,I_k$ are $\nabla$-closed and $E$-closed.
\end{lemma}
\begin{proof}
1. $I$ is $\nabla$-closed: for $x=n\circ \tilde{x}\in \Gamma(I)$
and any vector field $y$,
$$\nabla_y \,x= \nabla_y (n\circ \tilde{x})
\stackrel{\eqref{condition-nabla-n}}{=}n\circ \nabla_y \,\tilde{x} ~\in ~\Gamma(I)~.$$
2. $J_k$ is $\nabla$-closed:
$$
x\in \Gamma(J_k)~~\Leftrightarrow~~n^k\circ x=0~~\Rightarrow~~
n^k\circ \nabla_y\,x\stackrel{\eqref{condition-nabla-n}}{=}\nabla_y(n^k\circ x)=0
~~\Leftrightarrow~~\nabla_y\, x~\in~ \Gamma(J_k)~.
$$
3. $I$ is $E$-closed: for $x=n\circ \tilde{x}\in \Gamma(I)$, we have
$$
[E,x]=[E,n\circ \tilde{x}]\stackrel{\eqref{E1}}{=}
[E,n]\circ \tilde{x}+n\circ[E,\tilde{x}]+n\circ\tilde{x}
\stackrel{\eqref{condition-E-n}}{=}n\circ ([E,\tilde{x}]+\tilde{x})~\in ~\Gamma(I)~.
$$
4. $J_k$ is $E$-closed: 
by \eqref{E1}, \eqref{condition-E-n} and the induction on $k$,
we can show that
\begin{equation} \label{E2}
[E,n^k]=(k-1)n^k~\qquad (1\leq k\leq d)~.
\end{equation}
We have
$$
x\in \Gamma(J_k)~~\Leftrightarrow~~n^k\circ x=0~~\Rightarrow~~
n^k\circ [E,x]\stackrel{\eqref{E1}}{=}[E,n^k\circ x]-[E,n^k]\circ x-n^k\circ x
\stackrel{\eqref{E2}}{=}
0~.
$$
\end{proof}

Define the bilinear form $(~,~)_k$ on $I_k/I_{k-1}$ 
by \eqref{nilp2}. Then we have a Frobenius filtration of subbundles
$(I_{\bullet},\,(~,~)_{\bullet}\,)$ on $TM$.

\begin{theorem}\label{thm:nilpotent2}
If $(\circ,\langle~,~\rangle,E)$ is a Frobenius structure on $M$
of charge $D$ with a nilpotent global vector field $n$
satisfying \eqref{condition-E-n} and \eqref{condition-nabla-n},
then
the Levi--Civita connection $\nabla$,
the Euler vector field $E$, the multiplication $\circ$
and the Frobenius filtration $(I_{\bullet},\,(~,~)_{\bullet}\,)$
(defined in Eq.\eqref{nilp1} and Definition \ref{nilp2})
form
a mixed Frobenius structure on $M$ of  reference charge $D+1$.
\end{theorem}

\begin{proof}
We have to check the conditions \eqref{eta-const}, \eqref{c-symmetry},
\eqref{E-multiplication} and \eqref{E-metric}.

We first show the case $k\geq 1$.
For $x,y\in \Gamma(I_k/I_{k-1})$ $(k\geq 1)$, 
take representatives $\tilde{x},\tilde{y}\in \Gamma(J_k)$.
Let $z,w $ be vector fields on $M$. 

Eq. \eqref{eta-const} :
\begin{equation}\nonumber
\begin{split}
z(x,y)_k=&z\langle \tilde{x},\tilde{y}\circ n^{k-1} \rangle
\\
=&\langle \nabla_z\,\tilde{x}, \tilde{y}\circ n^{k-1} \rangle
+\langle \tilde{x},\nabla_z (\tilde{y}\circ n^{k-1})\rangle
\\
\stackrel{\eqref{condition-nabla-n}}{=}&
\langle \nabla_z\,\tilde{x}, \tilde{y}\circ n^{k-1} \rangle
+\langle \tilde{x},(\nabla_z\, \tilde{y})\circ n^{k-1}\rangle
\\
=&(\nabla_z^{(k)}\, x,y)_k+(x,\nabla_z^{(k)}\,y)_k~.
\end{split}
\end{equation}

Eq. \eqref{c-symmetry} :
\begin{equation}\nonumber
\begin{split}
(\nabla_w c_k)(x,y,z)
=&w(x,~y\circ_k \pi_k(z))_k-(\nabla_w^{(k)}\,x,~y\circ_k\pi_k(z))_k
\\
&-(x,~\pi_k(z)\circ_k \nabla_w^{(k)}\,y)_k-(x,~y\circ_k \pi_k(\nabla_w \,z))_k
\\
=&w\langle \tilde{x},~\tilde{y}\circ z \circ n^{k-1}\rangle
-\langle \nabla_w\, \tilde{x},~\tilde{y}\circ n^{k-1}\circ z \rangle
\\
&-\langle \tilde{x},~z\circ n^{k-1}\circ 
\nabla_w\,\tilde{y}\rangle 
-\langle \tilde{x},~\tilde{y}\circ n^{k-1}\circ \nabla_w \,z \rangle
\\
\stackrel{\eqref{condition-nabla-n}}{=}&(\nabla_w \,c)(\tilde{x},n^{k-1}\circ \tilde{y},z)\stackrel{\eqref{c1}}{=}(\nabla_z \,c)(\tilde{x},n^{k-1}\circ \tilde{y},w)
\\
=&(\nabla_z c_k)(x,y,w)~.
\end{split}
\end{equation}

Eq. \eqref{E-multiplication}: by \eqref{E1}, we have
\begin{equation}\nonumber
[E,\tilde{x}\circ z]-\tilde{x}\circ [E,z]
-[E,\tilde{x}]\circ z=\tilde{x}\circ z~.
\end{equation}
Applying the projection $\pi_k: TM\to TM/I_{k-1}$ to the both sides,
we obtain \eqref{E-multiplication}.

Eq. \eqref{E-metric} :
\begin{equation}\nonumber
\begin{split}
&E\langle \tilde{x},\tilde{y}\circ n^{k-1} \rangle-
\langle\, [E,\tilde{x}],\tilde{y}\circ n^{k-1}\rangle-
\langle x,[E,\tilde{y}]\circ n^{k-1}
\rangle
\\ \stackrel{\eqref{E1}}{=}&
E\langle \tilde{x},\tilde{y}\circ n^{k-1} \rangle-
\langle\, [E,\tilde{x}],\tilde{y}\circ n^{k-1}\rangle-
\langle \tilde{x},\,[E,n^{k-1}\circ \tilde{y}]-\tilde{y}\circ [E,n^{k-1}]
-\tilde{y}\circ n^{k-1}\rangle
\\
\stackrel{\eqref{E3}}{=}
&(2-D)\langle \tilde{x},\tilde{y}\circ n^{k-1}\rangle
+\langle \tilde{x},\tilde{y}\circ [E,n^{k-1}]\rangle
+\langle \tilde{x},\tilde{y}\circ n^{k-1}\rangle
\\
\stackrel{\eqref{E2}}{=}&
(1-D+k)\langle \tilde{x},\tilde{y}\circ n^{k-1}\rangle~.
\end{split}
\end{equation} This implies
\begin{equation}\nonumber
E(x,y)_k-
(\, [E,x]_k,y)_k-
( x,[E,y]_k)_k=(1-D+k)(x,y)_k~.
\end{equation}

Next we show the case $k=0$. Let $x=n\circ\tilde{x},~y\in \Gamma(I)$. 
Recall that $(x,y)_0=\langle \tilde{x},y\rangle$.

Eq. \eqref{eta-const} : since $\nabla_z\, x=\nabla_z (n\circ \tilde{x})=n\circ \nabla_z\, \tilde{x} $ by \eqref{condition-nabla-n},
we have
\begin{equation}\nonumber
(\nabla^{(0)}_z\, x, y)_0+(x,\nabla^{(0)}_z \,y)_0
=\langle \nabla_z \,\tilde{x},y \rangle
+\langle \tilde{x},\nabla_z\,y\rangle
=z\langle \tilde{x},y\rangle=
z(x,y)_0~.
\end{equation}

Eq. \eqref{c-symmetry} : since $\nabla_z\, x=n\circ \nabla_z\, \tilde{x} $,
\begin{equation}\nonumber
\begin{split}
(\nabla_w c_0)(x,y,z)
=&w(x,~y\circ z)_0-(\nabla^{(0)}_w\,x,~y\circ z)_0
-(x, ~z\circ \nabla^{(0)}_w\,y)_0-(x,~y\circ \nabla^{(0)}_w \,z)_0
\\
\stackrel{\eqref{condition-nabla-n}}{=}
&w\langle \tilde{x}, ~y\circ z \rangle
-\langle \nabla_w\, \tilde{x}, ~y\circ z\rangle
-\langle \tilde{x},~z\circ \nabla_w\,y\rangle 
-\langle \tilde{x},~y\circ \nabla_w \,z\rangle
\\
=&(\nabla_w \,c)(\tilde{x},y,z)
\stackrel{\eqref{c1}}{=}(\nabla_z \,c)(\tilde{x},y,w)
\\
=&(\nabla_z\, c_0)(x,y,z)~.
\end{split}
\end{equation}

Eq. \eqref{E-multiplication} follows immediately from \eqref{E1}.

Eq. \eqref{E-metric} : 
notice that
$[E,x]=n\circ ([E,\tilde{x}]+\tilde{x})$ holds by \eqref{E1}.
Therefore
\begin{equation}\nonumber
\begin{split}
E( x,y)_0-( [E,x],y)_0-
( x,\,[E,y])_0
=&E\langle \tilde{x},y\rangle-\langle [E,\tilde{x}]+\tilde{x} ,y\rangle
-\langle \tilde{x},\,[E,y]\rangle
\\
\stackrel{\eqref{E3}}{=}&(2-D)\langle \tilde{x},y\rangle
-\langle \tilde{x},y\rangle
\\=&(1-D)( x,y)_0~.
\end{split}
\end{equation}
\end{proof}

\subsection{Quotient construction}
Next we apply Lemma \ref{lem:quotient-construction2} to the 
mixed Frobenius structure obtained in Theorem \ref{thm:nilpotent2}
with $k=0$.

\begin{corollary}\label{quotient-construction2}
If $(\circ,\langle~,~\rangle,E)$ is a Frobenius structure on $M$
of charge $D$ with a nilpotent global vector field $n$
satisfying \eqref{condition-E-n} and \eqref{condition-nabla-n},
then
$(\nabla^{(1)},~E^{(1)},~\circ_1, ~I_{\bullet}/I_0,~(~,~)_{\bullet})$ is a
mixed Frobenius structure of reference charge $D+1$ on $M^{(1)}$.
\end{corollary}

Examples of this construction can be found in  \S \ref{example-toricsurface} and \S \ref{example-localP3}.
\section{Local quantum cohomology of weak Fano toric surfaces}
\label{example-toricsurface}
Let $S$ be a weak Fano toric surface.
We define the local quantum cup product
on the cohomology $H^*(S,\mathbb{C})$
using genus zero 
local Gromov--Witten invariants
and  construct
a mixed Frobenius structure.

First we recall basic facts about the quantum cohomology 
in \S \ref{QH}.
Then we state the results in \S \ref{toric-results}.
In \S \ref{H-V} 
and \S \ref{Proof-prop-toric3}
we explain that they can be obtained 
from the quantum cohomology of
 the projective compactification $V$ of
the canonical bundle $K_S$.

\subsection{Quantum cohomology}
\label{QH}

Let $V$ be a nonsingular projective variety.
$\overline{M}_{g,n}(V,\beta)$
denotes the moduli stack of genus $g$, $n$-pointed stable maps
to $V$ of degree $\beta\in H_2(V,\mathbb{Z})$.
Its virtual dimension is 
\begin{equation}\label{virt-dim}
(1-g)(\dim V-3)-\int_{\beta}c_1(K_V)+n~,
\end{equation}
where $K_V$ is the canonical bundle of $V$.
Let
$ev_i:\overline{M}_{g,n}(V,\beta)\to V$ be the evaluation map
at the $i$-th marked point.

Fix a basis of the even part $H^{\rm even}(V,\mathbb{Q})$ 
of the cohomology $H^*(V,\mathbb{Q})$:
\begin{equation}\label{basis-X}
\Gamma_0=1~,~\underbrace{ \Gamma_1,\ldots,\Gamma_r}_{H^2(V)}~,
~\underbrace{\Gamma_{r+1},\ldots,\Gamma_{s}}_{H^{\geq 4}(V)}~.
\end{equation}
The dual basis with respect to the intersection form 
is denoted $\{\Gamma_i^{\vee}\}$, i.e.
$$
\int_V \Gamma_i\cup \Gamma_j^{\vee} =\delta_{i,j}~.
$$
Let $t^0,\ldots,t^s$ be the coordinates of 
$H^{\rm even}(V,\mathbb{Q})$ associated to
\eqref{basis-X}.

\begin{definition}
The genus zero Gromov--Witten potential of $V$ is defined by
\begin{equation}\label{GW-potential}
\begin{split}
\Phi(t\,;q)&=
\sum_{n=0}^{\infty}\sum_{\beta\in H_2(V,\mathbb{Z})}
\frac{1}{n!}
\Big(\int_{[\overline{M}_{0,n}(V,\beta)]^{vir}}
\prod_{i=1}^n\, ev_i^* \mathbf{t}~\Big) q^{\beta},
\\
\mathbf{t}&=\sum_{i=0}^s t^i\Gamma_i~.
\end{split}
\end{equation}
Here $q$ is the parameter associated to $H_2(V,\mathbb{Z})$
and $[\overline{M}_{0,n}(V,\beta)]^{vir}$ is the virtual fundamental class.
\end{definition}
Recall that
the contribution $\Phi_{cl}$ from $\beta=0$ in $\Phi(t\,;q)$
is given by the triple intersection because of the point mapping axiom
(see \cite[\S 2]{KM}, also \cite[Chapter 8]{Cox-Katz}, for axioms of Gromov--Witten invariants):
\begin{equation}
\Phi_{cl}=\sum_{i,j,k=0}^s \frac{t^it^jt^k}{3!}\,\int_V
\Gamma_i\cup \Gamma_j\cup\Gamma_k~.
\label{GW-classical}
\end{equation}

\begin{definition}
The quantum cup product $\circ_t$ is defined by
\begin{equation}\label{quantum-cup}
\Gamma_i\circ_t\Gamma_j=\sum_{k=0}^s
\frac{\partial^3 \Phi}{\partial t^i\partial t^j\partial t^k}\,\Gamma_k^{\vee}~
\quad (0\leq i,j\leq s)~.
\end{equation}
We call $(H^{\rm even}(V),\circ_t)$ the quantum cohomology of $V$.\footnote{
Usually the quantum cohomology refers to a superalgebra structure
on $H^*(V)$ 
and the quantum cohomology considered  here is 
its subalgebra.
For our purpose, this is sufficient 
because $H^*(V)=H^{\rm even}(V)$ for 
nonsingular toric varieties $V$ 
which we deal with in \S \ref{example-toricsurface}, 
\S \ref{example-localP3}.
}
\end{definition}

The intersection form $\langle~,~\rangle$ is defined by
\begin{equation}\label{X-intersection}
\langle \Gamma_i,\Gamma_j\rangle=\int_V \Gamma_i\cup \Gamma_j~.
\end{equation}
Then $(H^{\rm even}(V),\circ_t,\langle~,~\rangle)$ is a Frobenius algebra.

Define the vector field $E$ on $M=H^*(V,\mathbb{C})$ by
\begin{equation}\label{QH-Euler}
E=\sum_{i=0}^s \frac{2-\mathrm{deg}\Gamma_i}{2}\, t^i
\frac{\partial}{\partial t^i} 
+\sum_{i=1}^r 
\xi_i \frac{\partial}{\partial t^i}~.
\end{equation}
Here the numbers $\xi_i$ are coefficients of $\Gamma_i$ in $-c_1(K_V)$:
$$
-c_1(K_V)=\sum_{i=1}^r \xi_i\Gamma_i~.
$$
\begin{theorem} \label{QH-Frob-str}
$(\circ,\langle~,~\rangle,E)$ defined in \eqref{quantum-cup},
\eqref{X-intersection} and \eqref{QH-Euler} is a Frobenius structure on 
$M=H^{\mathrm{even}}(V,\mathbb{C})$
of charge $\dim V$ 
(\cite{KM}, see also \cite{Manin}).
\end{theorem}

\begin{remark}\label{rem:conv}
When $V$ is a smooth projective toric variety, 
the convergence of the Gromov--Witten potential \eqref{GW-potential} 
was proved by Iritani \cite[Theorem 1.3]{Iritani}. We only deal
with such cases in \S \ref{example-toricsurface} and
\S \ref{example-localP3}.
\end{remark}
\subsection{Results}
\label{toric-results}
Let $S$ be a weak Fano toric surface.
Let $\gamma_{r+1}\in H^4(S,\mathbb{Z})$ be the Poincar\'e dual of 
the point class and $\gamma_0=1\in H^0(S)$.
Let $\gamma_1,\ldots,\gamma_r$ be a basis of $H^2(S,\mathbb{Z})$.
Define the integers $c_{ij},b_i,b_{i}^{\vee},\kappa$ ($1\leq i,j\leq r$) by
\begin{equation}\begin{split}\nonumber
&c_{ij}=\int_S \gamma_i\cup\gamma_j~,\quad
-c_1(K_S)=\sum_{i=1}^r b_i\gamma_i~,\\
&b_i^{\vee}=\sum_{j=1}^rb_jc_{ji}~,\quad
\kappa=\int_S c_1(K_S)^2=\sum_{i,j=1}^r b_ib_jc_{ij}=\sum_{i=1}^r b_ib_i^{\vee}~,
\end{split}
\end{equation}
where $c_1(K_S)$ is the first Chern class of the canonical bundle $K_S$.
Notice that the matrix $(c_{ij})$ is invertible. 
Notice also that $\kappa> 0$ holds for weak Fano toric surfaces.

Consider the filtration\footnote{When $S$ is Fano, \eqref{toric-filtration} and \eqref{toric-bilinear} 
agree with those in Example \ref{example-surface}
with $n=K_S$ if the filtration is shifted by one.}
on $H^*(S,\mathbb{C})$ by subspaces
\begin{equation}\label{toric-filtration}
0= I_{0}\subset I_1=\mathbb{C}\,c_1(K_S)\oplus H^4(S)\subset I_2
=H^2(S)\oplus H^4(S)=I_3
\subset I_4=H^*(S)~,
\end{equation}
and the following bilinear forms on $I_k/I_{k-1}$:
\begin{equation}\label{toric-bilinear}
\begin{split}
&( c_1(K_S),\gamma_{r+1})_1=1~,\quad
( c_1(K_S),c_1(K_S))_1=( \gamma_{r+1},\gamma_{r+1})_1=0~,\\
&( \gamma_i,\gamma_j)_2=c_{ij}-\frac{b_i^{\vee}b_j^{\vee}}{\kappa}
\quad (1\leq i,j\leq r)~,\\
&(\gamma_0,\gamma_0)_4=\kappa~.
\end{split}
\end{equation}


Next we define the local quantum cup product.
Let $\overline{M}_{g,n}(S,\beta)$ 
be the moduli stack of $n$-pointed genus $g$ stable maps to $S$
of degree $\beta\in H_2(S,\mathbb{Z})$,
$ev_i:\overline{M}_{g,n}(S,\beta)\to S$ be the evaluation map
at the $i$-th marked point,
$\mu:\overline{M}_{g,1}(S,\beta)\to  \overline{M}_{g,0}(S,\beta)$
be the forgetful map.
\begin{definition}
For an effective class $\beta\in H_2(S,\mathbb{Z})$ satisfying
$-\int_\beta c_1(K_S)> 0$, 
we define
\begin{equation}\label{local-GW}
N_{\beta}=\int_{[\overline{M}_{0,0}(S,\beta)]^{vir}} e(R^1\mu_*ev_1^*K_S)~.
\end{equation}
Here $e$ stands for the Euler class.
For other $\beta\in H_2(S,\mathbb{Z})$, we just set $N_{\beta}=0$.
We call $N_{\beta}$ the genus zero local Gromov--Witten invariants
of degree $\beta$ of $S$. 
\end{definition}
These numbers can be found in \cite{CKYZ} for some $S$.
%

Let $\{C_1,\ldots,C_r\}$ be the basis of $H_2(S,\mathbb{Z})$
dual to $\gamma_1,\ldots,\gamma_r$.
\begin{definition}
The local quantum cup product on $H^*(S,\mathbb{C})$ is 
the following family of multiplications $\circ_t$
parameterized by $t=(t^1,\ldots,t^r)$
:
\begin{equation}\label{toric-multiplication}
\begin{split}
&\gamma_0\circ_t \gamma_i=\gamma_i~,\quad
\gamma_0\circ_t \gamma_{r+1}=\gamma_{r+1}~,\quad
\gamma_i\circ_t \gamma_{r+1}=\gamma_{r+1}\circ_t \gamma_{r+1}=0~,
\\
&\gamma_i\circ_t \gamma_j=\Big(c_{ij}-
\sum_{\begin{subarray}{c}
\beta=\sum_i \beta_iC_i\\
\end{subarray}}
\beta_i\,\beta_j (b\cdot \beta)\,N_{\beta}~
e^{\beta\cdot t}\Big)\, \gamma_{r+1}~
\quad (1\leq i,j\leq r)~,
\end{split}
\end{equation}
where $\beta\cdot t=\sum_{i=1}^r \beta_i t^i$
and $b\cdot \beta= \sum_{i=1}^r b_i \beta_i$.
\end{definition}

%
%
%

\begin{remark}
Although we defined the local quantum cup product by explicit formulas,
it is possible to define it in a uniform manner. One way is to use equivariant 
localization with respect to the $\mathbb{C}^*$-action rotating the fiber of the
canonical bundle. See e.g. \cite{KonishiMinabe14}. Another way is to use the fact 
that the evaluation map is proper in all the example in this paper. See 
\cite[\S1.4]{BryanGraber}. The formulas given in \eqref{toric-multiplication} 
coincide with these general definitions. The authors thank the referee for pointing out this fact. 
\end{remark}

Let $M=H^*(S,\mathbb{C})$ and $t^0,t^1,\ldots,t^{r+1}$
be the coordinates associated to the basis
$\gamma_0,\gamma_1,\ldots,\gamma_{r+1}$.
Define a vector field $E$ on $M$ by
$$
E=t^0\frac{\partial}{\partial t^0}-t^{r+1}\frac{\partial}{\partial t^{r+1}}~.
$$
Regard the multiplication \eqref{toric-multiplication},
the filtration \eqref{toric-filtration}
and the bilinear forms \eqref{toric-bilinear}
on those on the tangent space $T_t M$ by
identifying $T_tM$ with $H^*(S,\mathbb{C})$.

\begin{theorem}\label{toric3}
The trivial connection, the above $E$,
the multiplication \eqref{toric-multiplication}
and the Frobenius filtration \eqref{toric-filtration}
and \eqref{toric-bilinear}
form a mixed Frobenius structure on $M=H^*(S,\mathbb{C})$
of reference charge four. 
\end{theorem}

The proof of
Theorem \ref{toric3} 
will be given in  \S \ref{Proof-prop-toric3}, 
by applying the quotient construction (Corollary \ref{quotient-construction2})
to the quantum cohomology ring of the projective compactification
of $K_S$.

\begin{remark}\label{rem:localB}
In this example, the operator 
$\mathcal{V}=\nabla E-\frac{2-D}{2}=-\frac{\text{deg}}{2}+2$
is diagonalizable and eigenvalues are integers.
Therefore other than the Frobenius filtration $I_{\bullet}$,
we can consider the decreasing filtration $\mathcal{F}^{\bullet}$
on $H^*(S,\mathbb{C})$
defined by
\begin{equation}\nonumber
\mathcal{F}^{p}=\bigoplus_{p'\geq p} 
\text{ the eigenspace of $\mathcal{V}$ with eigenvalue $p'$}~.
\end{equation}
This $\mathcal{F}^{\bullet}$ is the same as the one in \S \ref{sec:MHS}.
Moreover $I_k~(\text{here})=\mathcal{W}_k\otimes \mathbb{C}$
with $n=c_1(K_S)$.
They agree with the Hodge filtration and the weight filtration
of the mixed Hodge structure 
for the corresponding local B-model
under an appropriate vector space isomorphism
(see \cite[Theorem 4.2]{KonishiMinabe08} and  references therein).
\end{remark}

\subsection{Quantum cohomology of 
$V=\mathbb{P}(\mathcal{O}_S\oplus K_S)$}\label{H-V}
Let $V=\mathbb{P}(\mathcal{O}_S\oplus K_S)$ be the projective
compactification of the canonical bundle $K_S$ of 
a weak Fano toric surface $S$ and 
$\mathrm{pr}:V\to S$ be the projection.
Let $\Gamma_i=\mathrm{pr}^*\gamma_i$ ($0\leq i\leq r+1$)
and let $\Delta_0$ be the Poincar\'e dual of 
the infinity section.
Set $\Delta_i=\Gamma_i\cup \Delta_0$ ($1\leq i\leq r$).
We have the following basis of $H^*(V,\mathbb{C})$:
\begin{equation}\label{V-basis}
\underbrace{\Gamma_0}_{H^0(V)}\quad
\underbrace{\Gamma_1,~\ldots,\Gamma_r,~\Delta_0}_{H^2(V)}\quad
\underbrace{\Gamma_{r+1},~\Delta_1,\ldots,~\Delta_r}_{H^4(V)}\quad 
\underbrace{\Delta_{r+1}}_{H^6(V)}~.
\end{equation}
Let $t_i,s_i$ ($0\leq i\leq r+1$) be the coordinates
of $H^{*}(V,\mathbb{C})$ associated to this basis.

The cup product $\cup$ can be  calculated by the intersection theory of 
toric varieties. The unit is $\Gamma_0$,  
$\Gamma_i\cup \Delta_0=\Delta_i$ by definition ($1\leq i\leq r+1$) 
and
\begin{equation}\label{cup-V}
\begin{split}
&\Gamma_i\cup\Gamma_j=c_{ij}\Gamma_{r+1}~,\quad
\Gamma_i\cup\Gamma_{r+1}=0~,\quad
\Gamma_i\cup\Delta_j=c_{ij}\,\Delta_{r+1},
\\
&\Delta_0^2=\sum_{i=1}^r b_i \Delta_i~,\quad
\Delta_0\cup \Delta_j=b_j^{\vee}\,\Delta_{r+1}~
\quad (1\leq i,j\leq r).
\end{split}
\end{equation}
Other products vanish by the degree reason.
The intersection form $\langle~,~\rangle$ can be 
obtained from $\int_V \Delta_{r+1}=1$ and the above cup product.
Explicitly pairings which do not vanish are:
\begin{equation}\label{intersection-V}
\begin{split}
&\langle \Gamma_0,\Delta_{r+1}\rangle=1~, \quad 
 \langle \Delta_0,\Gamma_{r+1}\rangle=1~,\\
&\langle \Gamma_i,\Delta_j\rangle =c_{ij}~,\quad
 \langle \Delta_0,\Delta_j\rangle =b_j^{\vee}~\quad (1\leq i,j\leq r)~.
\end{split}
\end{equation}

Let $A=(a_{ij})$ be the inverse of the matrix $(c_{ij})_{1\leq i,j\leq r}$.
The dual basis
of \eqref{V-basis} is given by the following.
\begin{equation}
\begin{split}
&\Gamma_0^{\vee}=\Delta_{r+1}~,\quad
\Gamma_{i}^{\vee}=\sum_{j=1}^ra_{ij}\Delta_j-b_i \Gamma_{r+1}~,
\quad
\Delta_0^{\vee}=\Gamma_{r+1}~,
\\
&\Gamma_{r+1}^{\vee}=\Delta_0-\sum_{k=1}^r b_k\Gamma_k~,
\quad
\Delta_i^{\vee}=\sum_{j=1}^r a_{ij}\Gamma_j~,
\quad
\Delta_{r+1}^{\vee}=\Gamma_0~.
\end{split}
\end{equation}

Now consider the Gromov--Witten potential $\Phi(t,s\,;q)$ 
(see \eqref{GW-potential})
and the quantum cup product \eqref{quantum-cup} of $V$.
To be concrete, let us  fix a basis of $H_2(V,\mathbb{Z})$
as follows.
Let
$\iota:S\rightarrow V$ be the inclusion as the zero 
section of $K_S\subset V$ and 
$C_i'=\iota_*C_i$ ($1\leq i\leq r$).
Let $C_{0}'$ be the  fiber class of $\mathrm{pr}:V\to S$.
Then $\{C_0',C_1',\ldots,C_r'\}$ is a base of $H_2(V,\mathbb{Z})$
which is dual to the basis $\{\Delta_0,\Gamma_1,\ldots,\Gamma_{r}\}$ 
of $H^2(V,\mathbb{Z})$.
Let $q_i$ ($0\leq i\leq r$) be  parameters
associated to $C_0',C_1',\ldots,C_{r}'$.
The parameter $q^{\beta}$ in \eqref{GW-potential} is written as
$$
q^{\beta}=\prod_{i=0}^r q_i^{\beta_i} \quad\text{ for }\quad
\beta=\sum_{i=0}^{r}\beta_iC_{i}'~.$$
\begin{lemma}\label{GW2}
We have
\begin{equation}
\begin{split}
\Phi(t,s\,;q)&=\Phi_{cl}+\Phi_0+\Phi_1~,\\
\Phi_{cl}
&=\frac{1}{2}\,t_0^2 s_{r+1}
+t_0\Big(\sum_{i,j=1}^r c_{ij}t_i s_j
+t_{r+1} s_0+\sum_{i=1}^r b_i^{\vee}s_0 s_i
\Big)
\\
&+\frac{1}{2}\sum_{i,j=1}^r c_{ij}t_it_js_0
+\frac{1}{2}\sum_{i=1}^{r}b_i^{\vee}t_i\, s_0^2~,
\\
\Phi_0&=
\sum_{\begin{subarray}{c}\beta=\beta_1C_1'+\cdots+\beta_r C_r'\neq 0
\end{subarray}}
N_{\beta}^V\,e^{\beta\cdot t} q^{\beta}~,\qquad 
\mathrm{where}\quad N_{\beta}^V=\int_{[\overline{M}_{0,0}(V,\beta)]^{vir}}\,1,
\\
\Phi_1&= \mathcal{O}(q_{0})~.
\end{split}
\end{equation}
\end{lemma}
\begin{proof}
We decompose the Gromov--Witten potential $\Phi$ into 
three parts: $\Phi_{cl}$ which is the contribution of $\beta=0$,
$\Phi_0$ which is the contribution of homology classes
$\beta\neq 0$ with $\beta_0=0$, and the remaining part $\Phi_{1}$
which is the  contribution of $\beta$ with $\beta_0\neq 0$.

$\Phi_{cl}$ can be computed by \eqref{GW-classical}.

In $\Phi_1$, the terms with $\beta_0<0$ vanish because
such $\beta$ are not effective and the moduli stack
$\overline{M}_{0,n}(V,\beta)$ is empty.
Therefore $\Phi_1=\mathcal{O}(q_0)$.

Notice that since $-c_1(K_V)=2\Delta_0$, 
the virtual dimension \eqref{virt-dim} of $\overline{M}_{0,n}(V,\beta)
$ is $2\beta_0+n$. 
If $\beta_0=0$,
\begin{equation}\nonumber
\begin{split}
\int_{[\overline{M}_{0,n}(V,\beta)]^{vir}}
\prod_{i=1}^n \,ev^*_i\mathbf{t}
&=\int_{[\overline{M}_{0,n}(V,\beta)]^{vir}}
\prod_{i=1}^n \,ev^*_i\mathbf{t'}~
\quad\Big(\mathbf{t'}=\sum_{i=1}^r t_i\Gamma_i +s_0\Delta_0\Big)
\\
&=\Big(\sum_{i=1}^r \beta_it^i\Big)^n
\int_{[\overline{M}_{0,0}(V,\beta)]^{vir}}\,1~.
\end{split}
\end{equation}
Here
the first equality follows from the degree consideration and 
the fundamental class axiom,
the second equality follows from the divisor axiom.
This proves the equation for $\Phi_0$.
\end{proof}

We consider
the specialization $q_{0}=0$, and set $q_1=\cdots=q_r=1$.\footnote{
Since $q_i$ ($1\leq i\leq r$)
appears in $\Phi_0$ always in the combination $e^{t^i}q_i$,
one can always recover $q_1,\ldots,q_r$ in $\Phi_0$. 
}
Then by Lemma \ref{GW2}, 
we see that the quantum cup product \eqref{quantum-cup} reduces to
\begin{equation}\label{quantum-cup-V}
\begin{split}
\Gamma_i\circ_t\Gamma_j&= c_{ij}\Gamma_{r+1}
+\sum_{k=1}^r 
\Big(\sum_{\begin{subarray}{c}
\beta=\beta_1C_1+\cdots+\beta_r C_r\\
\beta\neq 0\end{subarray}} \beta_i\,\beta_j\,\beta_k\, N^V_{\beta} \,
e^{\beta\cdot t}\Big)\,\Gamma_k^{\vee}
 \quad  (1\leq i,j\leq r)~,
\\
\Delta_i\circ_t *&=\Delta_i\cup *~.
\end{split}
\end{equation}

\begin{lemma}\label{FS-V}
The multiplication \eqref{quantum-cup-V},
the intersection form \eqref{intersection-V} 
(regarded as the multiplication on $T_tH^*(V,\mathbb{C})$
and the metric by the canonical isomorphism $H^*(V)\stackrel{\sim}{\rightarrow}T_tH^*(V)$)and 
the vector field
\begin{equation}\label{Euler-SV}
E=t^0\frac{\partial}{\partial t^0}+2\frac{\partial}{\partial s^0}
-t^{r+1}\frac{\partial}{\partial t^{r+1}}
-\sum_{i=1}^r s^i\frac{\partial}{\partial s^i}
-2s^{r+1}\frac{\partial}{\partial s^{r+1}}
\end{equation}
form a Frobenius structure of charge three 
on $\tilde{M}=H^{*}(V,\mathbb{C})$.
\end{lemma}

\begin{proof}
The quantum cup product on $V$
is a power series in 
$t^0, e^{t^i}q_i (1\leq i\leq r)$, $e^{s^0}q_0,$
$t^{r+1}$, $s^i(1\leq i\leq r+1)$ and it is convergent \cite{Iritani}.  
By Lemma \ref{GW2}, it becomes as follows.
\begin{equation}\nonumber
\begin{split}
\Gamma_i\circ_t\Gamma_j&= c_{ij}\Gamma_{r+1}
+\sum_{k=1}^r 
\Big(\sum_{\begin{subarray}{c}
\beta=\beta_1C_1+\cdots+\beta_r C_r\\
\beta\neq 0\end{subarray}} \beta_i\,\beta_j\,\beta_k\, N^V_{\beta} \,
e^{\beta\cdot t}\Big)\,\Gamma_k^{\vee}+\mathcal{O}(e^{s^0}q_0)
 \quad  (1\leq i,j\leq r)~,
\\
\Delta_i\circ_t *&=\Delta_i\cup *~+\mathcal{O}(e^{s^0}q_0).
\end{split}\end{equation}
The above multiplication \eqref{quantum-cup-V} is 
just the terms of degree zero in $e^{s^0}q_0$ of this product.
Therefore the associativity and the commutativity of \eqref{quantum-cup-V}
follow from those of this quantum cup product.

The Euler vector field \eqref{QH-Euler} for the quantum cohomology of $V$
is given by \eqref{Euler-SV} since $-c_1(K_V)=2\Delta_0$.
The compatibility \eqref{E1} with the multiplication \eqref{quantum-cup-V}
follows from that of the quantum cohomology.

The symmetry \eqref{c1} of $\nabla c$ holds because of the same reason.
\end{proof}
\subsection{Proof of Theorem \ref{toric3}}
\label{Proof-prop-toric3}

We first apply 
Theorem \ref{thm:nilpotent2}
to the Frobenius structure on $\tilde{M}=H^{*}(V,\mathbb{C})$ 
obtained in Lemma \ref{FS-V}
with the nilpotent vector field $n=\frac{\partial}{\partial s^0}$
which satisfies the conditions \eqref{condition-E-n}
and \eqref{condition-nabla-n}.

Let us construct the Frobenius filtration. 
Using the canonical isomorphism $T_t\tilde{M}\cong H^*(V,\mathbb{C})$,
we write it down as that of $H^*(V,\mathbb{C})$.
Let $I$ be the ideal generated by $\Delta_0$:
\begin{equation}\nonumber
I=\mathbb{C} \Delta_0\oplus \bigoplus_{i=1}^{r}\mathbb{C}\Delta_i
 \oplus \mathbb{C}\Delta_{r+1}~.
\end{equation}
Then following the construction in \S \ref{nilpotent-construction}, we compute
$J_k=\mathrm{Ker}\,\Delta_0^k$:
\begin{equation}\nonumber
\begin{split}
J_1&=\mathbb{C}\Gamma_{r+1}^{\vee}\oplus 
   \bigoplus_{i=1}^{r}\mathbb{C}\Gamma_i^{\vee}
  \oplus \mathbb{C}\Delta_{r+1}~,
\\
J_2&=\bigoplus_{i=1}^r \mathbb{C}(\Gamma_i-\frac{b_i^{\vee}}{\kappa}\Delta_0)
  \oplus  \mathbb{C}\Gamma_{r+1}\oplus
  \bigoplus_{i=1}^{r}\mathbb{C}\Delta_i\oplus \mathbb{C}\Delta_{r+1}
~,
\\
J_3&=\bigoplus_{i=1}^{r} \mathbb{C}\Gamma_i\oplus \mathbb{C}\Delta_0
\oplus \mathbb{C}\Gamma_{r+1}
\bigoplus_{i=1}^{r}\mathbb{C}\Delta_i\oplus 
 \mathbb{C}\Delta_{r+1}=H^{\geq 2}(V)\supset I~,
\\
J_4&=\mathbb{C} \Gamma_0\oplus J_3=H^*(V)~.
\end{split}
\end{equation}
So $I_k=I+J_k$ ($k=1,2,3,4$) are
 \begin{equation}\label{toric-filtration2}
\begin{split}
I_0&=I~,\\
I_1&=\mathbb{C}\Gamma_{r+1}^{\vee}\oplus \mathbb{C}\Delta_0
    \oplus \mathbb{C}\Gamma_{r+1}
\oplus
\bigoplus_{i=1}^r \mathbb{C}\Delta_i\oplus \mathbb{C}\Delta_{r+1},\\
I_2&=I_3=H^{\geq 2}(V)~,\\
I_4&=H^*(V)~.
\end{split}
\end{equation}
The induced bilinear forms $(~,~)_k$ on $I_k/I_{k-1}$ are as follows.
\begin{equation}\label{toric-bilinear2}
\begin{split}
k=0\quad&
(\Delta_0,\Delta_{r+1})_0=\langle \Delta_0,\Gamma_{r+1}\rangle=1~,\\
&(\Delta_i,\Delta_j)_0=\langle\Delta_i,\Gamma_j\rangle=c_{ij}\quad (1\leq i,j\leq r)~,
\\
k=1\quad &
([\Gamma_{r+1}^{\vee}],[\Gamma_{r+1}])_1
=\langle \Gamma_{r+1}^{\vee},\Gamma_{r+1}\rangle
=1~,
\\
&([\Gamma_{r+1}^{\vee}],[\Gamma_{r+1}^{\vee}])_1=
([\Gamma_{r+1}],[\Gamma_{r+1}])_1=0~,
\\
k=2\quad &
([\Gamma_i],[\Gamma_j])_2
=
\Big\langle \Big(\Gamma_i-\frac{b_i^{\vee}}{\kappa}\Delta_0\Big)\circ_t
\Big(\Gamma_j-\frac{b_j^{\vee}}{\kappa}\Delta_0\big),~ \Delta_0 \Big\rangle
\\&\qquad\quad\quad=c_{ij}-\frac{b_i^{\vee}b_j^{\vee}}{\kappa}~\quad (1\leq i,j\leq r),
\\
k=4\quad&
( [\Gamma_0],[\Gamma_0])_4=\langle \Gamma_0\cup \Gamma_0, \Delta_0^3\rangle 
=\kappa~.
\end{split}
\end{equation}

Thus by Theorem \ref{thm:nilpotent2}, we have the following lemma.
\begin{lemma}\label{toric4}
The trivial connection, the vector field $E$ \eqref{Euler-SV},
the multiplication \eqref{quantum-cup-V},
the filtration \eqref{toric-filtration2} 
and the bilinear forms \eqref{toric-bilinear2}
form a MFS of reference charge four on $\tilde{M}=H^*(V,\mathbb{C})$~.
\end{lemma}

Next we apply Corollary \ref{quotient-construction2}.
Since 
$I$ is the kernel of 
the pullback
$\iota^*: H^*(V,\mathbb{C})\to H^*(S,\mathbb{C})$ 
by the inclusion $\iota:S\hookrightarrow K_S\subset V$,
if we set $$
\tilde{M}^{(1)}=\{s^0=s^1=\cdots=s^{r+1}=0\}\subset \tilde{M}~,
$$
then it is naturally isomorphic to $H^*(S,\mathbb{C})$.
Theorem \ref{toric3} follows from 
Corollary \ref{quotient-construction2} and Lemma \ref{FS-V}
together with the fact that $N^V_{\iota_*\beta}=N_{\beta}$
for $\beta \in H_2(S, \mathbb{Z})$ (see e.g. \cite[Proposition 2.2]{KonishiMinabe}).

\begin{remark}
It is not difficult to see that the argument in \S \ref{H-V} and \S \ref{Proof-prop-toric3} shows that the same 
result as in Theorem \ref{toric3} holds for any smooth projective surface $S$ with nef 
anticanonical bundle. (For such an $S$, the quantum cup products \eqref{toric-multiplication} 
and \eqref {quantum-cup-V} make sense  without any change as formal power series.) 
So we have a formal MFS on the even part of
$H^*(S, \mathbb{C})$. If $S$ is toric, then we have a genuine 
MFS since the quantum cup product is known to be convergent (cf. Remark \ref{rem:conv}).
This is the only point where we used the toric assumption on $S$.   
\end{remark}

\section{Local quantum cohomology of $\mathbb{P}^3$}
\label{example-localP3}
In this section, 
we  construct a mixed Frobenius structure on the cohomology of the 
projective space $\mathbb{P}^3$
similar to the one in \S \ref{example-toricsurface}.

\subsection{Results}
Take the following basis of the cohomology $H^*(\mathbb{P}^3,\mathbb{C})$:
$$
\gamma_0=1~,\quad
\gamma_1=c_1(\mathcal{O}_{\mathbb{P}^3}(1))~,\quad
\gamma_2=\gamma_1\cup\gamma_1~,\quad
\gamma_3=\gamma_1\cup \gamma_2~.
$$
Let $t^0,t^1,t^2,t^3$ be the associated coordinates.
We identify $H_2(\mathbb{P}^3,\mathbb{Z})$ with $\mathbb{Z}$
by $\beta \mapsto \int_\beta \gamma_1$.

We consider the following filtration\footnote{\eqref{P3-filtration} and \eqref{P3-bilinear} 
agree with those in Example \ref{example-Pn} with $m=-4$
if the filtration is shifted by one.} on $H^*(\mathbb{P}^3)$.
\begin{equation}\label{P3-filtration}
\begin{split}
I_0&=0~,\\
I_1&=\cdots=I_{4}=H^{\geq 2}(\mathbb{P}^3)~,
\\
I_{5}&=H^*(\mathbb{P}^3)~.
\end{split}
\end{equation}
On the graded quotients, we consider the bilinear forms:
\begin{equation}\label{P3-bilinear}
\begin{split}
I_1/I_0~:~&\quad (\gamma_k,\gamma_l)_1=\begin{cases}
-\frac{1}{4}&(k+l=4)\\
0&(k+l\neq 4)~
\end{cases}~,
\\
I_{5}/I_{4}~:~&\quad
(1,1)_{5}=4^3~.
\end{split}
\end{equation}

Next we define the local quantum cup product.
Let $\overline{M}_{g,n}(\mathbb{P}^3,\beta)$ 
be the moduli stack of $n$-pointed genus $g$ stable maps to $\mathbb{P}^3$
of degree $\beta\in H_2(\mathbb{P}^3,\mathbb{Z})\cong \mathbb{Z}$,
$ev_i:\overline{M}_{g,n}(\mathbb{P}^3,\beta)\to \mathbb{P}^3$ 
be the evaluation map
at the $i$-th marked point and
$\mu:\overline{M}_{g,2}(\mathbb{P}^3,\beta)\to  
\overline{M}_{g,1}(\mathbb{P}^3,\beta)$
be the forgetful map.

\begin{definition}
For $\beta\neq 0$, we define the 
genus zero local Gromov--Witten invariant $N_{\beta}\in\mathbb{Q}$ 
of degree $\beta\in H_2(\mathbb{P}^3,\mathbb{Z})$ by
 \begin{equation}\label{local-GW-P3}
N_{\beta}:=\int_{[\overline{M}_{0,1}(\mathbb{P}^3,\beta)]^{vir}}
ev_1^*\gamma_2\cup e(R^1 \mu_*ev_2^*\mathcal{O}_{\mathbb{P}^3}(-4))~.
\end{equation}
\end{definition}

These numbers are computed in \cite[Table 1]{KlemmPandharipande}.

\begin{definition}
Let 
\begin{equation}\label{local-pot-P3}
\Phi_{qu}(t^1,t^2):=t^2 \sum_{\beta>0}N_{\beta}\, e^{\beta t^1}~.
\end{equation}
The local quantum cup product $\circ_t$ on $H^*(\mathbb{P}^3,\mathbb{C})$ 
is the family of multiplications given by
\begin{equation}\label{P3-multiplication1}
\gamma_i\circ_t \gamma_j=\gamma_i \cup \gamma_j
-4\sum_{k=1}^2\frac{\partial^3 \Phi_{qu}}{\partial t^i\partial t^j\partial t^k}\,\gamma_{4-k}~.
\end{equation}
\end{definition}
More explicitly, the local quantum cup product is given by
\begin{equation}\label{P3-multiplication2}
\begin{split}
\gamma_0\circ_t&=\gamma_0\cup~,\quad \gamma_3\circ_t=\gamma_3\cup~,\\
\gamma_1\circ_t\gamma_1&=
\gamma_2-4\Big(\sum_{\beta>0}\beta^2 N_{\beta}\,e^{\beta t^1}\Big)\gamma_2
 -4\Big(t^2\sum_{\beta>0}\beta^3 N_{\beta}\,e^{\beta t^1}\Big)\gamma_3~,\\
\gamma_1\circ_t\gamma_2&=\gamma_3
-4\Big(\sum_{\beta>0}\beta^2 N_{\beta}\,e^{\beta t^1}\Big)\gamma_3~,\\
\gamma_2\circ_t\gamma_2&=0~.
\end{split}
\end{equation}


Regard $M=H^*(\mathbb{P}^3,\mathbb{C})$ as a manifold.
We identify each tangent space $T_t M$ with $H^*(\mathbb{P}^3,\mathbb{C})$
by $\frac{\partial}{\partial t^i}\mapsto \gamma_i$.
Let $\nabla$ be the trivial connection on $TM$
such that $\frac{\partial}{\partial t^i}$ are flat sections.
We set
\begin{equation}\label{Euler-localP3}
E=t^0\frac{\partial}{\partial t^0}-t^2\frac{\partial}{\partial t^2}-2
t^3\frac{\partial}{\partial t^3}.
\end{equation}
\begin{theorem}\label{prop-P3-2}
The trivial connection $\nabla$, the vector field $E$ \eqref{Euler-localP3},
the local quantum cup product \eqref{P3-multiplication2},
the filtration \eqref{P3-filtration} and the bilinear forms \eqref{P3-bilinear}
form a mixed Frobenius structure on $M$ of reference charge five.
\end{theorem}
A proof will be given in \S \ref{proof-prop-P3-2}.

\subsection{Quantum cohomology of  $V=\mathbb{P}(\mathcal{O}_{\mathbb{P}^3}
\oplus\mathcal{O}_{\mathbb{P}^3}(-4))$ }
Let $V=\mathbb{P}(\mathcal{O}_{\mathbb{P}^3}
\oplus\mathcal{O}_{\mathbb{P}^3}(-4))$ 
be the projective compactification of the line bundle 
$\mathcal{O}(-4)\to \mathbb{P}^3$.
Let $\mathrm{pr}: V\to \mathbb{P}^3$ be the projection
and let $\Gamma_i=\mathrm{pr}^*\gamma_i$ ($0\leq i\leq 3$).
Let $\Delta_0\in H^2(V,\mathbb{C})$ be the Poincar\'e dual of the 
infinity section 
and 
$\Delta_i=\Delta_0\cup \Gamma_i\in H^{2(i+1)}(V,\mathbb{C})$.
The cohomology $H^*(V,\mathbb{C})$ is spanned by the basis
\begin{equation}\label{basisV2}
\Gamma_0=1~,\quad \underbrace{\Gamma_1~,\Delta_0}_{H^2(V)}~,
\quad
\underbrace{\Gamma_2~,\Delta_1}_{H^4(V)}~,
\quad
\underbrace{\Gamma_3~,\Delta_2}_{H^6(V)}~,
\quad
\underbrace{\Delta_3}_{H^8(V)}~.
\end{equation}

The cup product is given as follows.
\begin{equation}\label{intersectionV2}
\begin{split}
\Gamma_i\cup \Gamma_j&=\begin{cases}\Gamma_{i+j}&(i+j\leq 3)\\
                                   0&(i+j>3)\end{cases}~,\quad
\\
\Delta_i\cup \Gamma_j&=\begin{cases} \Delta_{i+j}&(i+j\leq 3)\\
                                    0&(i+j>3)\end{cases}~,\quad
\\
\Delta_i\cup\Delta_j&=\begin{cases}4\Delta_{i+j+1} &(i+j<3)\\
                                    0&(i+j\geq 3)\end{cases}~.
\end{split}
\end{equation}
The intersection form $\langle~,~\rangle$ is computed from the cup product
and $\int_V \Delta_3=1$. Explicitly, pairings which do not vanish are
\begin{equation}\label{intersection-V2}
\begin{split}
&\langle \Gamma_0,\Delta_3\rangle=1~,\quad
\langle \Gamma_k,\Delta_{3-k}\rangle=1~,
\\
&\langle \Delta_{k-1},\Delta_{3-k}\rangle=4~\quad 
(1\leq k\leq 3)~.
\end{split}
\end{equation}
The dual basis of \eqref{basisV2} is given by
\begin{equation}\nonumber
\begin{split}
&\Gamma_0^{\vee}=\Delta_3~,\quad
\Delta_3^{\vee}=\Gamma_0~,
\\
&\Gamma_k^{\vee}=\Delta_{3-k}-4\Gamma_{4-k}~,\quad
{\Delta^{\vee}_{k-1}}=\Gamma_{4-k}\quad (1\leq k\leq 3)~.
\end{split}
\end{equation}

Now we consider the quantum cup product.
First we  fix a basis of $H_2(V,\mathbb{Z})$. 
Let
$\iota:\mathbb{P}^3\rightarrow V$ be the inclusion as the zero 
section of $\mathcal{O}_{\mathbb{P}^3}(-4)\subset V$ and 
$C_1'=\iota_*\mathbb{P}^1$.
Let $C_{0}'$ be the  fiber class of the $\mathbb{P}^1$-bundle 
$\mathrm{pr}:V\to \mathbb{P}^3$.
Then $\{C_0',C_{1}'\}$ is a basis of $H_2(V,\mathbb{Z})$
which is dual to the basis $\{\Delta_0,\Gamma_1\}$ 
of $H^2(V,\mathbb{C})$.
Let $t^0,t^1,s^0,t^2,s^1,t^3,s^2,s^3$
be the coordinates of $H^*(V)$ associated with the basis \eqref{basisV2}.
Let $q_0,q_1$ be  parameters
associated to $C_0',~C_{1}'$.

\begin{lemma}\label{lemP3}
The genus zero Gromov--Witten potential $\Phi(t,s\,;q)$ 
of $V$ becomes as follows.
\begin{equation}\nonumber
\begin{split}
\Phi(t,s\,;q)&=\Phi_{cl}+\Phi_{qu}+\Phi_1~,
\\
\Phi_{cl}(t,s)&=
\frac{1}{2}\sum_{\begin{subarray}{c} 0\leq k,l\leq 3,\\k+l\leq 3 \end{subarray}}
t^kt^ls^{3-k-l}
+\frac{4}{2}
\sum_{\begin{subarray}{c}0\leq k,l\leq 3,\\k+l\leq 2 \end{subarray}}
t^{2-k-l}s^ks^l~,
\\
\Phi_1&=\mathcal{O}(q_{0})~.
\end{split}
\end{equation}
Here $\Phi_{qu}$ is the function in $t^1,t^2$ defined in  \eqref{local-pot-P3}.
\end{lemma}

\begin{proof}
Notice that the class $\beta\in H_2(V,\mathbb{Z})$ of an effective curve  can be written as  
$\beta=\beta_0C_0'+\beta_1C_1'\neq 0$ with
$\beta_0,\beta_1\geq 0$.
Since the moduli of the stable maps $\overline{M}_{0,n}(V,\beta)$ ($\beta\neq 0$)
is empty if $\beta$ is not effective, 
we can decompose the Gromov--Witten potential $\Phi$ into 
three parts: $\Phi_{cl}$ which is the contribution of $\beta=0$,
$\Phi_0$ which is the contribution of homology classes
$\beta$ such that  $\beta=\beta_1 C_1'$ ($\beta_1>0$), 
and the remaining part $\Phi_{1}$
which is the  contribution of $\beta=\beta_0 C_0'+\beta_1 C_1'$ 
($\beta_0,\beta_1\geq 0$)
with $\beta_0\neq 0$.

$\Phi_{cl}$ can be computed by the triple intersection 
(see \eqref{GW-classical}).

We have $\Phi_1=\mathcal{O}(q_0)$.

Now we consider $\Phi_0$.
Notice that since $-c_1(K_V)=2\Delta_0$, 
the virtual dimension \eqref{virt-dim} of $\overline{M}_{0,n}(V,\beta)
$
is $1+n$ if $\beta_0=0$.
Therefore we have
\begin{equation}\nonumber
\begin{split}
\int_{[\overline{M}_{0,n}(V,\beta)]^{vir}}
\prod_{i=1}^n \,ev^*_i\mathbf{t}
&=\sum_{j=1}^n \int_{[\overline{M}_{0,n}(V,\beta)]^{vir}}
ev^*_j \mathbf{t}''\,
\prod_{\begin{subarray}{c}1\leq i\leq n;\\i\neq j\end{subarray}}
\,ev^*_i\mathbf{t}'~
\\
&=n (\beta_1 t^1)^{n-1}
\int_{[\overline{M}_{0,1}(V,\beta)]^{vir}}ev^*_1 \mathbf{t}'' 
\\
&=n(\beta_1 t^1)^{n-1} t^2
\int_{[\overline{M}_{0,1}(V,\beta)]^{vir}}ev^*_1 \Gamma_2~
\\
&=n(\beta_1 t^1)^{n-1} t^2\,N_{\beta_1}~,
\\
&\Big(\mathbf{t}'=t^1\Gamma_1 +s^0\Delta_0~,\quad 
\mathbf{t}''=t^2\Gamma_2+s^1 \Delta_1\Big)~.
\end{split}
\end{equation}
Here $N_{\beta_1}$ 
is the genus zero local Gromov--Witten invariant 
of $\mathbb{P}^3 $ defined in \eqref{local-GW-P3}.
The first equality follows from the degree consideration and 
the fundamental class axiom,
the second equality follows from the divisor axiom.
The third and fourth equalities follow from the next lemma.
This shows that $\Phi_0=\Phi_{qu}$.
\end{proof}
\begin{lemma}
For $\beta_1>0$, 
\begin{equation}\begin{split}\nonumber
&
\int_{[\overline{M}_{0,1}(V,\,\beta_1 C_1')]^{vir}}ev^*_1 \Delta_1=0~,
\\
&\int_{[\overline{M}_{0,1}(V,\,\beta_1 C_1')]^{vir}}ev^*_1 \Gamma_2
=N_{\beta_1}~.
\end{split}
\end{equation}
\end{lemma}
\begin{proof}
Let us consider the $\mathbb{C}^*$-action in the fiber direction
of $V$ and do the 
localization calculation \cite{GraPa}.
The fixed point loci is
$$
\overline{M}_{0,1}(V,\beta_1 C_1')^{\mathbb{C}^*}=
\iota (\overline{M}_{0,1}(\mathbb{P}^3,\beta_1\mathbb{P}^1))~.
$$
Here we use $\iota$ also as the map
$\overline{M}_{0,1}(\mathbb{P}^3,\beta_1\mathbb{P}^1)\to
\overline{M}_{0,1}(V,\beta_1 C_1')
$
induced from the inclusion $\iota:\mathbb{P}^3\to V$
as the zero section of $\mathcal{O}(-4)$.
Therefore we have
\begin{equation}\nonumber
\int_{[\overline{M}_{0,1}(V,\beta_1 C_1']^{vir}}ev^*_1\Delta_1
=
\int_{[\overline{M}_{0,1}(\mathbb{P}^3,\beta_1\mathbb{P}^1)]^{vir}}
\iota^*(ev_1^*\Delta_1)\cup e(R^1 \mu_*ev_2^*\mathcal{O}_{\mathbb{P}^3}(-4))
~.
\end{equation}
Here $e(R^1 \mu_*ev_2^*\mathcal{O}_{\mathbb{P}^3}(-4))$
is the contribution of the normal bundle of the fixed loci
(see \cite[Proposition 2.2]{KonishiMinabe}).
Since the commutativity of the diagram
$$
\begin{CD}
\overline{M}_{0,1}(\mathbb{P}^3,\beta_1\mathbb{P}^1)@>\iota>>\overline{M}_{0,1}(V,\beta_1C_1)\\
@Vev_1VV @Vev_1VV\\
\mathbb{P}^3@>\iota>>V
\end{CD}
$$
implies $\iota^*ev_1^*\Delta_1=ev^*_1 \iota^*\Delta_1=0$,
we obtain the first statement.
The proof of the second statement is similar.
\end{proof}

Let us consider the specialization $q_0=0$, $q_1=1$.
Then by Lemma \ref{lemP3},
the quantum cup product reduces to the following:
\begin{equation}\label{quantum-cup-V2}
\begin{split}
\Gamma_k\circ\Gamma_l&=\Gamma_{k+l}
+\sum_{j=1}^{2}~\frac{\partial^3 \Phi_{qu}}{\partial t^k \partial t^l\partial t^j}~(\Gamma_{j})^{\vee}~,
\\
\Delta_i\,\circ ~~&=\Delta_i\,\cup~~~.
\end{split}
\end{equation}

\begin{lemma}\label{P3-V}
The multiplication \eqref{quantum-cup-V2},
the intersection form \eqref{intersection-V2} 
(regarded as the multiplication on $T_t\tilde{M}$
and the metric by the canonical isomorphism $H^*(V)\stackrel{\sim}{\rightarrow}T_t\tilde{M}$) and 
the vector field
\begin{equation}\label{Euler-P3V}
E=t^0\frac{\partial}{\partial t^0}
-\sum_{k=2}^3 \,(k-1)\Big(t^k\frac{\partial}{\partial t^k}
 +s^{k-1}\frac{\partial}{\partial s^{k-1}}\Big)
-3s^3\frac{\partial}{\partial s^3}
+2\frac{\partial}{\partial s^0}~
\end{equation}
form a Frobenius structure of charge four on $\tilde{M}=H^{*}(V,\mathbb{C})$.
\end{lemma}
The proof of the lemma is similar to that of Lemma \ref{FS-V} and omitted.
\subsection{Proof of Theorem \ref{prop-P3-2}}
\label{proof-prop-P3-2}

We first apply 
Theorem \ref{thm:nilpotent2}
to the Frobenius structure on $\tilde{M}=H^*(V,\mathbb{C})$ in Lemma \ref{P3-V}
with the nilpotent vector field $\frac{\partial}{\partial s^0}$.


Let us construct the Frobenius filtration. 
Using the canonical isomorphism $T_t\tilde{M}\cong H^*(V,\mathbb{C})$,
we write it down as that of $H^*(V,\mathbb{C})$.
Let $I$ be the ideal generated by $\Delta_0$:
$$
I=\mathbb{C}\Delta_0\oplus\mathbb{C}\Delta_1\oplus
\mathbb{C}\Delta_2\oplus \mathbb{C}\Delta_3~\subset~H^*(V).
$$
$J_k=\mathrm{Ker}\,(\Delta_0^k\,\cup)$ are as follows.
\begin{equation}
\begin{split}
J_1&=
\mathbb{C}\Gamma_3^{\vee}\oplus \mathbb{C}\Gamma_2^{\vee}\oplus
\mathbb{C}\Gamma_1^{\vee}
\oplus \mathbb{C}\Delta_3~,
\\
J_2&=J_1+ (\mathbb{C}\Delta_{2}\oplus \mathbb{C}\Gamma_3)
=\mathbb{C}\Gamma_3^{\vee}\oplus \mathbb{C}\Gamma_2^{\vee}
\oplus H^{\geq 6}(V)~,
\\
J_3&=\mathbb{C}\Gamma_{3}^{\vee}\oplus H^{\geq 4}(V)~,
\\
J_4&=H^{\geq 2}(V)~,
\\
J_5&=H^*(V)~.
\end{split}
\end{equation}
Therefore the filtration $I_{\bullet}$ defined in \eqref{nilp1} on 
$H^*(V,\mathbb{C})$ is given by 
\begin{equation}\label{filtration-V2}
I_0=I~,\quad
I_k=H^{\geq 2 }(V)~\quad (1\leq k\leq 4)~,\quad 
I_{5}=H^*(V)~.
\end{equation}
The bilinear forms on $I_k/I_{k-1}$
(see Definition \eqref{nilp2}) are
\begin{equation}\label{bilinear-V2}
\begin{split}
k=0\quad&([\Delta_k],[\Delta_l])_0=
\langle \Delta_k,\Gamma_l\rangle=\delta_{k+l,3}~,
\\
k=1\quad &([\Gamma_k],[\Gamma_l])_1=\Big\langle 
\Gamma_k-\frac{1}{4}\Delta_{k-1},
\Gamma_l-\frac{1}{4}\Delta_{l-1}
\Big\rangle
\\
&\quad\quad\quad\quad
=\begin{cases}
-\frac{1}{4}&(k+l=4)\\
0&(k+l\neq 4)
\end{cases},
\\
k=5\quad& (1,1)_{5}=\langle 1 ,\Delta_0^{4}\rangle=4^3~.
\end{split}
\end{equation}

By Theorem \ref{thm:nilpotent2}, we have the following 
\begin{lemma}
The trivial connection, the vector field $E$ \eqref{Euler-P3V},
the multiplication \eqref{quantum-cup-V2},
the filtration \eqref{filtration-V2} and the bilinear forms \eqref{bilinear-V2}
form a MFS of reference charge five on $\tilde{M}=H^*(V,\mathbb{C})$~.

\end{lemma}




Next we apply Corollary \ref{quotient-construction2}.
Since 
$I$ is the kernel of 
the pullback
$\iota^*: H^*(V,\mathbb{C})\to H^*(\mathbb{P}^3,\mathbb{C})$ 
by the inclusion $\iota:\mathbb{P}^3 \hookrightarrow \mathcal{O}(-4)\subset V$,
if we set $$
\tilde{M}^{(1)}=\{s^0=s^1=s^2=s^3=0\}\subset \tilde{M}~,
$$
then it is naturally isomorphic to $H^*(\mathbb{P}^3,\mathbb{C})$.
Theorem \ref{prop-P3-2} follows from 
Corollary \ref{quotient-construction2} and Lemma \ref{P3-V}.


\appendix

\section{Deformed connection}
\label{section:deformed-connection}
In this appendix, we define an analogue of the deformed connection
of the Frobenius structure \cite{Dubrovin2} for the MFS. Let 
$
(\nabla,~E,~\circ,~I_{\bullet},~(~,~)_{\bullet})
$
be a MFS on $M$ of reference charge $D$ 
and let $t^{ka}$ $(k\in \mathbb{Z},~1\leq a\leq m_k)$
be a system of local flat coordinates satisfying \eqref{flat-frame}.

\subsection{Operators}
Define endomorphisms $\mathcal{U},\mathcal{V}:TM\rightarrow TM$ by
\begin{equation}
\mathcal{U}(x)=E\circ x~,\quad
\mathcal{V}(x)=\nabla_x E -\frac{2-D}{2}\,x~.
\end{equation}

\begin{lemma} If $x\in \Gamma(I_k)$, then
$\mathcal{U}(x)\, ,\, \mathcal{V}(x)\in \Gamma(I_k)$.
\end{lemma}
\begin{proof} If $x\in \Gamma(I_k)$,
$U(x)=E\circ x\in \Gamma(I_k)$ 
since $I_k$ is an ideal.

If $x\in \Gamma(I_k)$, we have
$$
\nabla_x E=[E,x]-\nabla_E\, x~\in\Gamma(I_k) ~,
$$
by the torsion free condition for $\nabla$ and 
the assumptions that $I_k$ is $E$-closed and $\nabla$-closed.
\end{proof}

The above lemma implies that $\mathcal{U}$, $\mathcal{V}$
induce endomorphisms $\mathcal{U}^{(k)}$, $\mathcal{V}^{(k)}$
on $I_k/I_{k-1}$.
In the local flat coordinate expression,
\begin{equation}\nonumber
\begin{split}
\mathcal{U}(\partial_{ka})
&=\sum_{\begin{subarray}{c}l\in \mathbb{Z},\\1\leq b\leq m_l\end{subarray}}
\sum_{\begin{subarray}{c}j\leq k,l,\\1\leq c\leq m_j\end{subarray}}
E^{lb}C_{ka,lb}^{jc} \partial_{jc}~,
\\
\mathcal{U}^{(k)}(\partial_{ka})
&=\sum_{\begin{subarray}{c}l \geq k,\\1\leq b\leq m_l\end{subarray}}
\sum_{1\leq c\leq m_k}
E^{lb}C_{ka,lb}^{kc} \partial_{kc}~,
\\
\mathcal{V}(\partial_{ka})&=
\sum_{\begin{subarray}{c}l\leq k ,\\1\leq b\leq m_l\end{subarray}}
(\partial_{ka}E^{lb})\partial_{lb}-\frac{2-D}{2}\partial_{ka}~,
\\
\mathcal{V}^{(k)}(\partial_{ka})&=
\sum_{1\leq b\leq m_k}
(\partial_{ka}E^{kb})\partial_{kb}-\frac{2-D}{2}\partial_{ka}~.
\end{split}
\end{equation}

\begin{remark}
The assumption $\nabla\nabla E=0$ implies $\nabla \mathcal{V}=0$.
In other words, the matrix representations of $\mathcal{V}$ and $\mathcal{V}^{(k)}$
with respect to the flat basis $\{\partial_{ka}\}$ are constant matrices.
Notice also that 
the condition
 \eqref{E-metric} is equivalent to 
\begin{equation}
( \mathcal{V}^{(k)}(x),y)_k+
( x,\mathcal{V}^{(k)}(y))_k=k( x,y)_k~.
\end{equation}
\end{remark}

\subsection{Deformed connection}
Let $\Tilde{M}=M\times \mathbb{C}^*$ and
let $\hbar$ be the coordinate of $\mathbb{C}^*$.
For a holomorphic vector bundle $\mathbb{E}\to \Tilde{M}$,
$\Tilde{\Gamma}(\mathbb{E})$ denotes 
the space of holomorphic sections of $\mathbb{E}$ on some open subset 
$\Tilde{U}\subset \Tilde{M}$.

Recall that 
$\pi_k:TM\to TM/I_{k-1}$ is the projection
and that $\nabla^{(k)}$ and $\circ_k$ 
are the connection and the multiplication on $TM/I_{k-1}$
induced from the connection $\nabla$ and the multiplication $\circ$~. 

\begin{definition}
Define a connection 
$\Tilde{\nabla}^{(k)}$ on $I_k/I_{k-1}\times T\mathbb{C}^*
\rightarrow \tilde{M}$ by
\begin{equation}
\begin{split}
&\Tilde{\nabla}^{(k)}_x y=\nabla^{(k)}_x\,y+\hbar\,\, \pi_k(x)\circ_k y
\quad (x\in\Tilde{\Gamma}(TM),~y\in \Tilde{\Gamma}(I_k/I_{k-1}))~,
\\
&\Tilde{\nabla}^{(k)}_{\hbar} y=\partial_{\hbar}\, y
+\mathcal{U}^{(k)}(y)+\frac{1}{\hbar}\,(\mathcal{V}^{(k)}(y)-\frac{k}{2}y\,)~,
\\
&\Tilde{\nabla}^{(k)}_x (\partial_{\hbar})=
\Tilde{\nabla}^{(k)}_{\hbar}(\partial_{\hbar})=0~.
\end{split}
\end{equation}
\end{definition}

We write $\Tilde{\nabla}_{la}^{(k)}$ 
for $\Tilde{\nabla}_{x}^{(k)}$ with $x=\partial_{la}$.
In the local flat coordinate expression,
\begin{equation}\label{deformed-conn1}
\begin{split}
\Tilde{\nabla}_{lb}^{(k)} (\partial_{ka})&=
\hbar\sum_{1\leq c\leq m_k}C_{ka,lb}^{kc} \,\partial_{kc}
\quad (\text{this is zero if $l<k$ by \eqref{c-symmetry2}})~,
\\
\Tilde{\nabla}^{(k)}_{\hbar}(\partial_{ka})
&=\sum_{\begin{subarray}{c}l\geq k,\\
1\leq b\leq m_l\end{subarray}}
\sum_{1\leq c\leq m_k}
(E^{lb}C_{ka,lb}^{kc})\partial_{kc}
\\
&+\frac{1}{\hbar}\,
\Big(\sum_{1\leq b\leq m_k}
(\partial_{ka}E^{kb})\partial_{kc}
-\frac{2-D+k}{2}\partial_{ka}\Big)~.
\end{split}
\end{equation}

\begin{proposition}
The deformed connection $\Tilde{\nabla}^{(k)}$ is flat.
\end{proposition}
\begin{proof}
Let 
$
\tilde{\Omega}^{(k)}$ 
be the curvature of $\Tilde{\nabla}^{(k)}$.
We first show that $\Tilde{\Omega}^{(k)}(\partial_{la},\partial_{jb})=
\frac{1}{2}(\Tilde{\nabla}^{(k)}_{la}\Tilde{\nabla}^{(k)}_{jb}
-\Tilde{\nabla}^{(k)}_{jb}\Tilde{\nabla}^{(k)}_{la})
=0$. 
By the first equation in \eqref{deformed-conn1},
it is immediate to check that
$\Tilde{\Omega}^{(k)}(\partial_{la},\partial_{jb})(\partial_{kc})=0
$ if $l,j<k$. If $l<k$ and $j\geq k$,
\begin{equation}\nonumber
\Tilde{\Omega}^{(k)}(\partial_{la},\partial_{jb})(\partial_{kc})=
\Tilde{\nabla}_{la}^{(k)}\Big(\sum_{1\leq d\leq m_k} 
C_{jb,kc}^{kd}\partial_{kd}\Big)
=\sum_{1\leq d\leq m_k} 
(\partial_{la}C_{jb,kc}^{kd}) \partial_{kd}
  \stackrel{\eqref{c-symmetry2}}{=}0~.
\end{equation}
If $l,j\geq k$,
\begin{equation}\nonumber
\begin{split}
\Tilde{\Omega}^{(k)}(\partial_{la},\partial_{jb})(\partial_{kc})
&=\Tilde{\nabla}_{la}^{(k)}\Big(\sum_{1\leq d\leq m_k} 
C_{jb,kc}^{kd}\partial_{kd}\Big)
-(la~\leftrightarrow~jb)
\\
&=\sum_{1\leq d\leq m_k}
(\partial_{la}C_{jb,kc}^{kd}-\partial_{jb}C_{la,kc}^{kd}) \partial_{kd}
\\&+
\hbar \sum_{1\leq d,f\leq m_k}(C_{jb,kc}^{kd}C_{la,kd}^{kf}-
C_{la,kc}^{kd}C_{jb,kd}^{kf})\,\partial_{kf}
=0~,
\end{split}
\end{equation}
by \eqref{c-symmetry2} and the associativity.

Next we show $\Tilde{\Omega}^{(k)}(\partial_{la},\partial_{\hbar})=
\frac{1}{2}(\Tilde{\nabla}^{(k)}_{la}\Tilde{\nabla}^{(k)}_{\hbar}
-\Tilde{\nabla}^{(k)}_{\hbar}\Tilde{\nabla}^{(k)}_{la})
=0$.
By the second equation in \eqref{deformed-conn1},
\eqref{E-linear} and \eqref{c-symmetry2},
if $l<k,$ we have
\begin{equation}\nonumber
\begin{split}
\Tilde{\Omega}^{(k)}(\partial_{la},\partial_{\hbar})(\partial_{kc})
&=\sum_{\begin{subarray}{c}j\geq k,\\
1\leq b\leq m_j\end{subarray}}
\sum_{1\leq d\leq m_k}
\big(\partial_{la}(E^{jb}C_{kc,jb}^{kd})\big)\partial_{kd}
=0~. 
\end{split}
\end{equation}
If $l\geq k$, we have
\begin{equation}\nonumber
\begin{split}
\Tilde{\nabla}^{(k)}_{la}\Tilde{\nabla}^{(k)}_{\hbar}
(\partial_{kc})
&=
\sum_{1\leq d\leq m_k}
\Big(\sum_{\begin{subarray}{c}j\geq k,\\
1\leq b\leq m_j\end{subarray}}
\partial_{la}(E^{jb}C_{kc,jb}^{kd})
+\sum_{1\leq b\leq m_k}
C_{la,kb}^{kd}\partial_{kc}E^{kb}-\frac{2-D+k}{2}C_{la,kc}^{kd}
\Big)\,
\partial_{kd}\\
&+\hbar\,
\sum_{1\leq f\leq m_k}
\Big(\sum_{1\leq d\leq m_k}
\sum_{\begin{subarray}{c}j\geq k,\\
1\leq b\leq m_j\end{subarray}}
E^{jb}C_{kc,jd}^{kd}C_{la,kd}^{kf}\Big) \,\partial_{kf}
~,
\\
\Tilde{\nabla}^{(k)}_{\hbar}\Tilde{\nabla}^{(k)}_{la}
(\partial_{kc})
&=\sum_{1\leq d\leq m_k}
\Big(C_{la,kc}^{kd}
+\sum_{1\leq b\leq m_k}C_{la,kc}^{kb}\partial_{kb}E^{kd}
 -\frac{2-D+k}{2}C_{la,kc}^{kd}
\Big)\,\partial_{kd}
\\
&+\hbar\,\sum_{1\leq f\leq m_k} 
\Big( 
\sum_{1\leq d\leq m_k}
\sum_{\begin{subarray}{c}j\geq k,\\
1\leq b\leq m_j\end{subarray}}
E^{jb} C_{la,kc}^{kd}C_{kd,jb}^{kf} 
\Big)\,\partial_{kf} ~.
\end{split}
\end{equation}
Thus by \eqref{E-multiplication2} and the associativity,  we obtain
\begin{equation}\nonumber
\begin{split}
\Tilde{\Omega}^{(k)}(\partial_{la},\partial_{\hbar})(\partial_{kc})
&=\Tilde{\nabla}^{(k)}_{la}\Tilde{\nabla}^{(k)}_{\hbar}
(\partial_{kc})
-\Tilde{\nabla}^{(k)}_{\hbar}\Tilde{\nabla}^{(k)}_{la}
(\partial_{kc})=0~.
\end{split}
\end{equation}
\end{proof}

\subsection{Deformed flat coordinates}
\begin{proposition}
There exist (local) holomorphic functions
$\tilde{t}^{ka}(t,\hbar) $ $(k\in \mathbb{Z}$, $1\leq a\leq m_k)$ 
on $\tilde{M}$
such that $\hbar$, $\tilde{t}^{ka}(t,\hbar)$ are
a system of local coordinates on $\tilde{M}$
satisfying the following conditions:
\begin{equation}\nonumber
\left \{
\frac{\partial}{\partial \tilde{t}^{la}} ~{\Big |}~ l\leq k, 1\leq a\leq m_l
\right\}\quad
\text{is a local frame of $I_k$,}
\end{equation}
\begin{equation}
\tilde{\nabla}^{(k)}\pi_k\Big(\frac{\partial}{\partial \tilde{t}^{ka}}\Big)
=0~.
\end{equation}
\end{proposition}
We call $\tilde{t}^{ka}(t,\hbar)$ deformed flat coordinates.

\begin{proof}
Let $\mathrm{Ann}_k\subset T^*M$ be the annihilator of $I_k$:
$$
\mathrm{Ann}_k:=\{x\in T^*M\mid x(y)=0~,
~~{}^{\forall}y\in I_k \}~.
$$
Its local frame is given by
$
dt^{la}~ (l>k).
$
Notice that
the dual bundle of $I_k/I_{k-1}$ is isomorphic to
$\mathrm{Ann}_{k-1}/\mathrm{Ann}_k$.
We use the same notation $\nabla^{(k)}$ for the induced dual connection on 
$\mathrm{Ann}_{k-1}/\mathrm{Ann}_k  \times T^*\mathbb{C}^*\to \tilde{M}$.

The $\nabla^{(k)}$-flatness condition $\Tilde{\nabla}^{(k)}\xi=0$ for
a section $\xi=\sum_{1\leq a\leq m_k}\xi_{ka}dt^{ka}\in 
\tilde{\Gamma}(\mathrm{Ann}_{k-1}/\mathrm{Ann}_k)$ is equivalent to
\begin{equation}\label{flat-dual}
\begin{split}
\partial_{lb}(\xi_{ka})&=
\hbar \sum_{1\leq c\leq m_k}C_{lb,ka}^{kc}\,\xi_{kc}\qquad \text{ and }
\\
\partial_{\hbar}(\xi_{ka})
&=
\sum_{\begin{subarray}{c}
l\geq k,\\1\leq b\leq m_l
\end{subarray}}
\sum_{1\leq c\leq m_k}
E^{lb}C_{lb,ka}^{kc}\,\xi_{kc}
+
\frac{1}{\hbar}\Big(
\sum_{1\leq c\leq m_k}
(\partial_{ka}E^{kc})\xi_{kc}-\frac{2-D+k}{2}\,\xi_{ka}
\Big)~.
\end{split}
\end{equation}
The first equation in \eqref{flat-dual} implies that
$
\partial_{kb}\xi_{ka}
=\partial_{ka}\xi_{kb}$ 
and
$\partial_{lb}\xi_{ka}=0
$ if $l<k$.
Therefore if $\Tilde{\nabla}^{(k)}\xi=0$, there exists a local function 
$\tilde{t}=\tilde{t}(t,\hbar)$ on $\tilde{M}$
satisfying 
$
\partial_{ka}\tilde{t}=\xi_{ka}
$ and 
$
\partial_{la}\tilde{t}=0
$ ($l<k$).
In other words, there exists $\tilde{t}(t,\hbar)$ such that
$$
d\tilde{t}=\xi+\text{ terms involving $dt^{lb}$ ($l>k$) and $d\hbar$}. 
$$

Since $\Tilde{\nabla}^{(k)}$ is flat,
there exists a local frame $\{p^{ka}(t,\hbar)\mid 1\leq a\leq m_k\}$
of 
$\mathrm{Ann}_{k-1}/\mathrm{Ann}_k  \boxtimes T^*\mathbb{C}^*\to \tilde{M}$
such that
$
\tilde{\nabla}^{(k)}p^{ka}=0
$.
From the argument in the previous paragraph,
 we see that there exist local functions
$\Tilde{t}^{ka}(t,\hbar)$ ($k\in \mathbb{Z}$, $1\leq a\leq m_k$)
satisfying the following two conditions:
\begin{equation}\nonumber
\begin{split}
&\{d\tilde{t}^{(la)}\mid l\geq k, 1\leq a\leq m_l\} \text{
is a local frame of $\mathrm{Ann}_{k-1}$},
\\
&p^{(ka)}=d\tilde{t}^{(ka)} \mod dt^{lb} ~~~(l>k),~d\hbar~.
\end{split}
\end{equation}
These $\Tilde{t}^{(ka)}(t,\hbar)$ satisfy the conditions
in the above proposition.
\end{proof}

\subsection{Deformed flat coordinates for
weak Fano toric surfaces}
\label{deformed-coord-toric}
The deformed flat coordinates for the MFS in Theorem \ref{toric3}
is written as follows.\footnote{We omit the detail of the calculation for the following results. 
It can be found in the first version of this paper at the arXiv.}

Assume that $b_r^{\vee}\neq 0$.
We take the following flat coordinates on $M=H^*(S,\mathbb{C})$ 
so that the condition
\eqref{flat-frame} is satisfied:
\begin{equation}\label{coord-S}
\begin{split}
&t^{r+1}~,\quad u^r=-\frac{1}{\kappa}\sum_{k=1}^r b_k^{\vee}\,t^k~,
\\
&u^k=\frac{1}{\kappa\cdot b_r^{\vee}}\Big\{
\sum_{\begin{subarray}{c}j\neq k,\\1\leq j\leq r\end{subarray}}
b_jb_j^{\vee}\,t^k
-b_k \sum_{\begin{subarray}{c}j\neq k,\\1\leq j\leq r\end{subarray}}
b_j^{\vee}\,t^j
\Big\}~\quad (1\leq k\leq r-1),\\
& t^0~.
\end{split}
\end{equation}

Solving the flatness equation for $\Tilde{\nabla}^{(k)}$
($1\leq k\leq 4$), we obtain 
the following deformed flat coordinates.
\begin{equation}\label{flat-coord-toric}
\begin{split}
&\tilde{t}^{r+1}=e^{\hbar\, t^0}\Big\{
\frac{1}{\sqrt{\hbar}}t^{r+1}+\sqrt{\hbar}
\Big(\frac{\kappa}{2}{u_r}^2
-\sum_{\beta\neq 0}(b\cdot \beta)N_{\beta}\,e^{\beta\cdot t}
\Big)
\Big\}~,
\\
&
{\tilde{u}}^r=\sqrt{\hbar}\,e^{\hbar\, t^0}u^r~,
\\
&\tilde{u}^k=e^{\hbar\,t^0}u^k\quad (1\leq k\leq r-1)~,
\\
&\tilde{t}^0=e^{\hbar\, t^0}~.
\end{split}
\end{equation}
\subsection{Deformed flat coordinates for $\mathbb{P}^3$ }
\label{deformed-coord-P3}
The deformed flat coordinates for the MFS in Theorem \ref{prop-P3-2}
is written as follows.
\begin{equation}\label{deformed-coord-P3-result}
\begin{split}
\tilde{t}^0&=\frac{1}{\hbar}e^{\hbar\,t^0}~,
\\
\tilde{t}^1&=\hbar\, e^{\hbar\,t^0}t^1~,
\\
\tilde{t}^2&=e^{\hbar\,t^0}\Big\{
t^2+\hbar\Big(\frac{(t^1)^2}{2}
-4\sum_{\beta>0} N_{\beta}\,e^{\beta\,t^1}\Big)\Big\}~,
\\
\tilde{t}^3&=e^{\hbar\,t^0}\Big\{
\frac{t^3}{\hbar}+
t^1t^2-4t^2\sum_{\beta>0}\beta N_{\beta}\,e^{\beta\,t^1}
\\
&+\frac{\hbar}{2}
\Big[\frac{(t^1)^3}{3}
-8\sum_{\beta>0}\Big(t^1-\frac{1}{\beta}\Big)N_{\beta}\,e^{\beta\,t^1}
+16\sum_{\beta,\gamma>0} \frac{\beta\gamma}{\beta+\gamma}
N_{\beta}N_{\gamma}e^{(\beta+\gamma)t^1}\Big]
\Big\}~.
\end{split}
\end{equation}

\end{document}